\documentclass[11pt]{amsart}

\usepackage{amsfonts, amssymb, amscd}

\usepackage{bm}
\usepackage{verbatim}
\usepackage{amssymb}
\usepackage{mathrsfs}
\usepackage{graphicx}
\usepackage[all]{xy}
\usepackage{tikz}
\usepackage{subcaption}
\usepackage{subfiles}
\usepackage[toc,page]{appendix}
\usepackage{mathtools}
\usepackage{comment}
\usepackage{enumerate}
\usepackage{enumitem}
\usepackage[normalem]{ulem}

\usepackage{graphicx}
\graphicspath{ {images/} }

\usepackage{appendix}
\usepackage{hyperref}

 \usepackage{todonotes}

\theoremstyle{plain}
 \newtheorem{thm}{Theorem}[section]
 \newtheorem{conj}[thm]{Conjecture}
 \newtheorem{cor}[thm]{Corollary}
\newtheorem{lem}[thm]{Lemma}
 
\newtheorem{prop}[thm]{Proposition}
\newtheorem{ques}[thm]{Question}
\theoremstyle{definition}
\newtheorem{defn}[thm]{Definition}
\newtheorem{rmk}[thm]{Remark}
\newtheorem{ex}[thm]{Example}

\theoremstyle{definition}
\newtheorem{definition}[thm]{Definition}

\newcommand{\FF}{\mathbb{F}}

\newcommand{\Pp}{\mathbb{P}}
\newcommand{\PP}{\mathbb{P}}
\newcommand{\Qq}{\mathbb{Q}}
\newcommand{\QQ}{\mathbb{Q}}
\newcommand{\Rr}{\mathbb{R}}
\newcommand{\RR}{\mathbb{R}}
\newcommand{\Zz}{\mathbb{Z}}
\newcommand{\ZZ}{\mathbb{Z}}

\newcommand{\Nn}{\mathbb{N}}
\newcommand{\NN}{\mathbb{N}}

\newcommand{\bfD}{\mathbf{D}}
\newcommand{\bfDiv}{\mathbf{Div}}
\newcommand{\bfM}{\mathbf{M}}

\newcommand{\Diff}{\operatorname{Diff}}
\newcommand{\Exc}{\operatorname{Exc}}

\newcommand{\Fr}{\operatorname{Fr}}
\newcommand{\glct}{\operatorname{glct}}

\newcommand{\Pic}{\operatorname{Pic}}

\newcommand{\lcm}{\operatorname{lcm}}

\newcommand{\Supp}{\operatorname{Supp}}

\newcommand{\me}{\mathcal{E}}

\newcommand{\mo}{\mathcal{O}}

\newcommand{\fe}{\mathfrak{e}}

\newcommand{\rmb}{\mathrm{b}}
\newcommand{\CDiv}{\mathrm{CDiv}}
\newcommand{\Div}{\mathrm{Div}}
\newcommand{\mult}{\mathrm{mult}}

\newcommand{\Nklt}{\mathrm{Nklt}}
\newcommand{\nonklt}{\mathrm{nonklt}}

\newcommand{\rmv}{\mathrm{v}}

\begin{document}

\title[nonvanishing for generalized polarized pairs]{On numerical nonvanishing for\\ generalized log canonical pairs}
\begin{abstract}
The nonvanishing conjecture for projective log canonical pairs plays a key role in the minimal model program of higher dimensional algebraic geometry. The numerical nonvanishing conjecture considered in this paper is a weaker version of the usual nonvanishing conjecture, but valid in the more general setting of  generalized log canonical pairs. We confirm it in dimension two. Under some necessary conditions we obtain effective versions of numerical nonvanishing for surfaces. Several applications are also discussed.

In higher dimensions, we mainly consider the conjecture for generalized klt pairs $(X, B+\bfM)$, and reduce it to lower dimensions when $K_X+\bfM_X$ is not pseudo-effective. 
Up to scaling the nef part, we prove the numerical nonvanishing for pseudo-effective generalized lc threefolds with rational singularities.
\end{abstract}

\author{Jingjun Han}
\address{Jingjun Han, Department of Mathematics, Johns Hopkins University, Baltimore, MD 21218, USA}
\email{jhan@math.jhu.edu}

\author{Wenfei Liu}
\address{Wenfei Liu, School of Mathematical Sciences, Xiamen University, Siming South Road 422, Xiamen, Fujian 361005, China}
\email{wliu@xmu.edu.cn}

\date{\today}

\maketitle

\tableofcontents

\section{Introduction}\label{sec: introduction}
The notion of generalized polarized pairs was introduced in \cite{BZ16} to deal with the effectivity of Iitaka fibrations. Its prototype already appeared in the treatment of the canonical bundle formulas \cite{K98}.  As it turns out, generalized polarized  pairs are natural objects in many more applications. For example, they are involved in the proofs of the BAB conjecture \cite{Bi19}, and Fujita's spectrum conjecture \cite{HL17}. In the meantime, Osamu Fujino showed that normal quasi-log canonical pairs can be given the structure of a generalized polarized pair \cite{Fu18}.

Some natural questions from the point of view of the minimal model program (MMP), such as the ascending chain condition for log canonical thresholds (\cite{BZ16}) and the canonical bundle formula (\cite{Fi18, HLiu19}), have been addressed in the setting of generalized polarized pairs. In this paper we consider a possible generalization of the following important conjecture:

\begin{conj}[Nonvanishing Conjecture]\label{conj: nonvanishing0}
Let $(X, B)$ be a projective log canonical pair such that $K_X+B$ is pseudo-effective. Then there exists an effective $\RR$-Cartier $\RR$-divisor $D$ such that $K_X+B\sim_{\RR} D$.
\end{conj}
The upshot (and difficulty) of Conjecture~\ref{conj: nonvanishing0} is that pseudo-effectivity is a numerical condition on the log canonical divisor $K_X+B$, while the conclusion is about, roughly speaking, the existence of global sections. Conjecture~\ref{conj: nonvanishing0} is one of the key conjectures in the minimal model program.

Easy examples indicate that one can only hope for a weaker version of the above conjecture in the setting of generalized polarized pairs. Basically, \emph{$\RR$-linear equivalence} should be replaced with the weaker \emph{numerical equivalence}. Note that numerical nonvanishing for pseudo-effective generalized log canonical pairs is enough to imply the existence of weak Zariski decomposition, which in turn guarantees the existence of minimal models (\cite{HM18,HL18}).

More precisely, we will focus on the following statement, which is an adjusted form of \cite[Question~3.5]{BH14}:
\begin{conj}[Numerical Nonvanishing for Generalized Polarized Pairs]\label{conj: nonvanishing}
Let $(X,B+\bfM)$ be a projective generalized log canonical pair. Suppose that
\begin{enumerate}
\item[(i)]   $K_X+B+\bfM_X$ is pseudo-effective,
\item[(ii)] $\bfM=\sum_j \mu_j \bfM_j$ where $\mu_j\in \Rr_{> 0}$ and $\bfM_j$ are nef Cartier $\rmb$-divisors.
\end{enumerate}
Then there exists an effective $\RR$-Cartier $\RR$-divisor $D$ such that $K_X+B+\bfM_X \equiv D$.
\end{conj}

The condition (i) of Conjecture~\ref{conj: nonvanishing} is of course necessary for the numerical nonvanishing; the following example shows that the condition (ii) cannot be ommited either:
\begin{ex}\label{rem: reasonNQC} 
Let $E$ be a general elliptic curve, and set $X=E\times E$. Consider the curves 
\[
F_1=\{P\}\times E ,F_2=E\times\{P\}, \Delta,
\]
where $P\in E$ is a fixed point and $\Delta\subset E\times E$ is the diagonal. According to \cite[Lemma 1.5.4]{L04}, $M=F_1+\sqrt{2}F_2+(\sqrt{2}-2)\Delta$ is nef, and it is not hard to show that for any $\epsilon>0$, there exists a curve $C$ satisfying $0<M\cdot C<\epsilon$; see \cite{vDdB18}. Suppose on the contrary that there exists an effective $\Rr$-Cartier $\Rr$-divisor $D$ such that $K_X+M\equiv D$. Then there exists a positive real number $\alpha>0$ such that $M\cdot C=D\cdot C\ge \alpha$ for any curve $C\nsubseteq\Supp D$, which is a contradiction.\end{ex}


\begin{rmk}
For projective generalized lc pairs $(X, B+\bfM)$ such that $K_X+B$ is $\RR$-Cartier and $\bfM_X$ is nef, Conjecture~\ref{conj: nonvanishing} relaxes the assumptions of the Generalized Nonvanishing Conjecture of \cite{LP18a} in the following aspects: (1) $K_X+B$ is not assumed to be pseudo-effective; (2) the singularities of $(X, B)$ are allowed to be lc but not klt; (3) the coefficients of $B$ and $\bfM_X$ are allowed to be real but not rational. Note that the assumptions of  \cite{LP18a} are necessary for the numerical abundance; see \cite[Section~6.1]{LP18a}. We give instead certain characterization for the failure of the numerical abundance under the (more general) conditions of Conjecture~\ref{conj: nonvanishing} in dimension two; see Section~\ref{sec: num abundance}.

The Generalized Nonvanishing Conjecture of  \cite{LP18a} holds if
\begin{itemize}
\item[(a)] the dimension of $X$ is two (\cite[Theorem~C]{LP18a}), or
\item[(b)] the dimension of $X$ is three and either the irregularity of the threefold or the numerical dimension of $K_X+B$ is positive (\cite[Corollary~D]{LP18a} and \cite[Theorem~8.1]{LP18b}). 
\end{itemize} 
Indeed, case (b) is a consequence of stronger conditional results that are valid in all dimensions; see \cite[Theorem~A]{LP18a} and \cite[Theorem~8.1]{LP18b}. Conjecture~\ref{conj: nonvanishing} is known to hold if the pair $(X, B)$ is klt with $\QQ$-coefficients and if $X$ admits a morphism $X\to Z$ to an abelian variety such that $K_X+B+M$ is big over $Z$ (\cite[Theorem~4.1]{BC15}). 
\end{rmk}

Our first main result is a confirmation of Conjecture~\ref{conj: nonvanishing} in dimension two.
\begin{thm}[= Theorem~\ref{thm: nv}]\label{thm: nonvanishing surface q}
Conjecture~\ref{conj: nonvanishing} holds true in dimension two.
\end{thm}

It can be used to prove the following two corollaries. 
\begin{cor}\label{cor: antinef}
	Let $(X,B)$ be a projective log canonical surface such that $-(K_X+B)$ is nef. Then there exists an effective $\Rr$-Cartier $\RR$-divisor $D$ such that $-(K_X+B)\sim_{\RR} D$. 
\end{cor} 
\begin{rmk}
In an earlier version of the paper, we deduced the numerical equivalence of $-(K_X+B)$ to an effective $\RR$-Cartier $\RR$-divisor as an easy consequence of Theorem~\ref{thm: nonvanishing surface q}. The $\RR$-linear equivalence as in Corollary~\ref{cor: antinef} was pointed out by the referee; it has been stated in \cite[Remark-Corollary 2.6]{Sh00}, but we could not find any proof thereof in the literature. The full proof given in Section 3 of our paper is indeed quite involved.
\end{rmk}

Recall that an $\RR$-Cartier $\RR$-divisor $L$ on a projective variety $X$ is \emph{strictly nef} if $L\cdot C>0$ for every curve $C$ on $X$. 
\begin{cor}\label{cor:serranosurface}
Let $(X,B)$ be a projective log canonical surface. Suppose $L$ is a strictly nef Cartier divisor on $X$. Then $K_X+B+tL$ is ample for every real number $t>3$. 
\end{cor}
In the special case of smooth surfaces with empty boundary, Corollary~\ref{cor:serranosurface} has been proved by Serrano (\cite[Proposition~2.1]{Se95}); see \cite{CCP08} and \cite[Section~6.3]{LP18a} for further development of the smooth case  in higher dimensions.

For generalized log canonical surfaces with bounded Gorenstein indices, we provide an effective version of Theorem~\ref{thm: nonvanishing surface q}.  It is an analogue of the classical statement of Enriques that $|12K_X|\neq \emptyset$ for a smooth projective surface that is not birational to a ruled surface.
\begin{thm}[=Theorem~\ref{thm: eff nv klt}]
Let $(X,B)$ be a projective log canonical surface and $M$ a nef $\Qq$-divisor on $X$ such that $r(K_X+B+M)$ is nef and Cartier for some $r\in \Nn$. Then $Nr(K_X+B+M)$
is numerically equivalent to an effective Cartier divisor for any integer $N\geq 3$.
\end{thm}

For numerically trivial generalized log canonical surfaces with coefficients from a given set of rational numbers satisfying the descending chain condition (DCC), we prove the existence of a uniform bound on the Cartier indices:
\begin{thm}[= Theorem~\ref{thm: dcc}]
Let $I\subset \QQ_{\geq 0}$ be a set satisfying the descending chain condition. Then there is an $N\in\NN$, depending only on $I$, such that if  $(X, B+\bfM)$ is a projective generalized lc surface with
\begin{enumerate}
\item[(i)] the coefficients of $B$ lying in $I$, 
\item[(ii)] $\bfM=\sum \mu_i \overline{M}_i$ where $M_i$ are nef Cartier divisors on $X$ and $\mu_i\in I$,
\item[(iii)] $K_X+B+\bfM_X\equiv 0$,
\end{enumerate}  
then $N(K_X+B+\bfM_X)$ is Cartier.
\end{thm}

In higher dimensions, we confirm Conjecture \ref{conj: nonvanishing} conditionally. Our typical assumption is that $K_X+\bfM_X$ is not pseudo-effective. In this case, one can use the technique of \cite[Proposition 8.7]{DHP13} to construct a Mori fibre space $X\dashrightarrow Y\rightarrow Z$, and then run appropriate MMPs over $Y$ or $Z$ to reach a generalized lc-trivial fibration, so that an induction on dimension can be applied. This has been exploited in \cite{Bi12, G15, DL15, Has17, Has18} for log canonical pairs $(X, B)$ with $K_X$ not pseudo-effective. 


\begin{thm}[= Theorem~\ref{thm: highdim}]\label{thm: highdim intro}
Assume that $\text{Conjecture~\ref{conj: nonvanishing}}$ holds for projective generalized klt pairs in dimensions less than $n$. 

Let $(X,B+\bfM)$ be an $n$-dimensional projective generalized klt pair such that $\bfM$ is an $\RR_{>0}$-linear combination of nef Cartier $\rmb$-divisors. Suppose that $K_X+B+\bfM_X$ is pseudo-effective and $K_X+\bfM_X$ is not pseudo-effective. Then $K_X+B+\bfM_X$ is numerically equivalent to an effective $\RR$-Cartier $\RR$-divisor.
\end{thm}

Assuming the termination of flips for $\Qq$-factorial dlt pairs, Theorem~\ref{thm: highdim intro} readily extends to the generalized lc pairs with rational singularities:
\begin{thm}[= Theorem~\ref{thm: highdim dlt}]\label{thm: highdim dlt0}
Assume that $\text{Conjecture~\ref{conj: nonvanishing}}$ holds in dimensions less than $n$. 
Assume that the termination of flips for $n$-dimensional projective $\Qq$-factorial dlt pairs holds. 

Let $(X,B+\bfM)$ be an $n$-dimensional projective generalized lc pair with rational singularities such that $\bfM$ is an $\RR_{>0}$-linear combination of nef Cartier $\rmb$-divisors. Suppose that $K_X+B+\bfM_X$ is pseudo-effective and $K_X+\bfM_X$ is not pseudo-effective. Then there exists an effective $\RR$-Cartier $\RR$-divisor $D$ such that $K_X+B+\bfM_X\equiv D$.
\end{thm}

The following corollary of Theoerem~\ref{thm: highdim dlt0} is inspired by the nonvanishing for pseudo-effective log canonical pairs of numerically trivial type, established in \cite{Has17}.
\begin{cor}[= Corollary~\ref{thm: CYnonvanishing threefold}]
	Let $(X, B+\bfM)$ be a projective generalized lc threefold with rational singularities such that $\bfM$ is an $\RR_{>0}$-linear combination of nef Cartier $\rmb$-divisors. Suppose that
	\begin{enumerate}[leftmargin=*]
		\item[(i)]$K_X+B+\bfM_X$ is pseudo-effective, and
		\item[(ii)] there exists a generalized lc pair $(X, C+\bfM)$ with $K_X+C+\bfM_X\equiv 0$.
	\end{enumerate}
	Then $K_X+B+\bfM_X$ is num-effective.
\end{cor}

Up to scaling the nef part, we are able to prove the numerical nonvanishing for projective generalized lc threefolds with rational singularities. 
\begin{thm}[= Theorem \ref{thm: dim3t}]\label{thm: scaleM}
	Let $(X,B+\bfM)$ be a projective generalized lc threefold with rational singularities such that $\bfM$ is an $\RR_{>0}$-linear combination of nef Cartier $\rmb$-divisors. If $K_X+B+\bfM_X$ is pseudo-effective and $\bfM_X$ is $\RR$-Cartier, then there exists a $0\le t\le 1$ such that $K_X+B+t\bfM_X$ is numerically equivalent to an effective $\RR$-Cartier $\RR$-divisor.
\end{thm}

The paper is organized as follows. In Section 2 we explain the relevant notions of divisors and define generalized polarized pairs and their singularities; a decomposition of nef (resp.~antinef) generalized log canonical divisors with $\RR$-coefficients into those with $\QQ$-coefficients is provided. In Section 3 we prove Conjecture~\ref{conj: nonvanishing} in dimension two and draw consequences thereof; several effective versions of the numerical nonvanishing are established. At the end of Section~\ref{sec: dim2} we discuss the (failure of) numerical abundance for generalized log canonical surfaces. In Section 4 we go to higher dimensions, restricting our attention mainly to generalized lc pairs with rational singularities;  after developing necessary tools for constructing generalized lc-trivial fibrations on generalized polarized pairs, we prove Theorems~\ref{thm: highdim intro} up to \ref{thm: scaleM}.

\medskip

\noindent{\bf Convention.}  Throughout the paper, we work over the complex numbers. We use the terminology in \cite{KM98} and \cite{Fu17} in general. For a set $A$ of real numbers, we use $A_{\geq 0}$ to denote the subset $\{a\in A \mid a\geq 0\}$. The subsets $A_{>0}$, $A_{\leq 0}$ and $A_{<0}$ are similarly defined. For an abelian group $G$ and a commutative ring $\FF$, we denote $G_{\FF} := G\otimes_{\ZZ} \FF$. 

\medskip
\noindent\textbf{Acknowledgements}.
Much of this work was done when the first author was visiting the second author at Xiamen University in June, 2018. The first author would like to thank Xiamen University for its hospitality.  
He would like to thank Christopher D.~Hacon, Chen Jiang and Roberto Svaldi for many useful discussions, and he also would like to thank his advisors Gang Tian and Chenyang Xu for constant support and encouragement. Part of this work was done during a visit of the second author to Universit\"at Bayreuth in August, 2018. He enjoyed the stimulating academic atmosphere there and wants to thank the members of Lehrstuhl Mathematik VIII for their hospitality. Upon finishing the first draft of this paper, we learned that Vladimir Lazi\'c and Thomas Peternell were working on a similar topic, independently. We want to thank them  for the consequent communication and helpful comments. Thanks also go to a referee who pointed out a gap in an earlier version of the paper, as well as a statement hidden in \cite{Sh00}, which becomes our Corollary~\ref{cor: antinef}. The second author was partially supported by the NSFC (No.~11971399, No.~11771294).
\section{Preliminaries}\label{sec: pre}
\subsection{Divisors} 
In this subsection, $X$ is a given normal variety equipped with a projective morphism $f\colon X\rightarrow S$. Let $\FF$ denote the integer ring $\ZZ$, the rational number field $\QQ$, or the real number field $\RR$. 
\begin{itemize}[leftmargin=*]
\item  An $\FF$-divisor is a finite formal $\FF$-linear combination $D = \sum_j d_j D_j $ of distinct prime Weil divisors $D_j$. The coefficients $d_j$ are also written as $\mult_{D_j} D$.  An $\FF$-divisor \emph{over} $X$ means  an $\FF$-divisor on a higher birational model $Y\rightarrow X$. We use $\sim_{\FF}$ to denote the $\FF$-linear equivalence between two $\FF$-divisors. An $\FF$-divisor is called an \emph{effective} (resp.~\emph{boundary}, resp.~\emph{subboundary}) divisor if its coefficients lie in $\FF_{\geq 0}$ (resp.~$[0,1]$, resp.~$\FF_{\leq 1}$).
 
 Note that a $\ZZ$-divisor is nothing but a Weil divisor and $\ZZ$-linear equivalence is just the usual linear equivalence. A $\ZZ$-divisor is Cartier if locally it can be defined by one rational function. The Weil divisors form a free abelian group, denoted by $\Div(X)$, and the subgroup of Cartier divisors is denoted by $\CDiv(X)$. For $\FF=\QQ$ or $\RR$, an $\FF$-Cartier $\FF$-divisor is an $\FF$-linear combination of Cartier divisors; an $\FF$-divisor (resp.~an $\FF$-Cartier $\FF$-divisor) can be viewed as an element of the $\FF$-vector space $\Div(X)_\FF$ (resp.~$\CDiv(X)_\FF$). 

\item In view of the projective morphism $f\colon X\rightarrow S$, we use $\sim_{\FF, f}$ or $\sim_{\FF, S}$ to denote the  $\FF$-linear equivalence between $\FF$-divisors, relative with respect to $f$.  Let $Z_1(X/S)$ be the free abelian group generated by integral curves that map to points on $S$. Then the intersection pairing $\CDiv(X) \times Z_1(X/S) \rightarrow \ZZ$, $(D, C) \mapsto D\cdot C$
induces a \emph{numerical equivalence} on both $\CDiv(X)$ and $Z_1(X/S)$, denoted by $\equiv_f$ or $\equiv_S$.  Set $N^1(X/S) = \CDiv(X)/\equiv_f$ and $N_1(X/S) = Z_1(X/S)/\equiv_f$. Then we obtain a perfect pairing of finite dimensional $\RR$-vector spaces
\[
N^1(X/S)_\RR\times N_1(X/S)_\RR\rightarrow \RR
\]
The \emph{cone of curves} $NE(X/S)\subset N_1(X)_\RR$ is the cone generated by numerical classes of integral curves $C$ in the fibres of $f\colon X\rightarrow S$; its closure is denoted by $\overline{NE}(X/S)$.

Let $D$ be an $\RR$-Cartier $\RR$-divisor on $X$. We say 
\begin{itemize}[leftmargin=*]
\item $D$ is \emph{nef/$S$} (read "nef over $S$") or \emph{$f$-nef} (resp.~\emph{antinef/$S$} or \emph{$f$-antinef}), if the intersection $D\cdot C\geq 0$ (resp.~$-D\cdot C\geq 0$) for any integral curve $C$ in a fibre of $f\colon X\rightarrow S$;
\item $D$ is \emph{pseudo-effective}/$S$ if its numerical class lies in the closure of the cone generated by the numerical classes of nef/$S$ $\RR$-divisors;
\item $D$ is \emph{num-effective}/$S$ if its numerical class lies in the effective cone $NE(X/S)$ (see \cite{LP18a});
\item $D$ is \emph{num-semiample}/$S$ if it is numerically equivalent to an semiample/$S$ $\RR$-divisor, the latter meaning a finite $\RR_{> 0}$-linear combination of semiample/$S$ Cartier divisors.
\end{itemize}


\end{itemize}
\begin{defn}
For $\FF=\ZZ, \QQ$ or $\RR$, a $\rmb$-$\FF$-divisor on $X$ is an element of the projective limit $${ \bfDiv} (X)_\FF = \lim_{Y\rightarrow X} \Div(Y)_\FF,$$ where the limit is taken over all proper birational morphism $\rho\colon Y\rightarrow X$ from a normal variety $Y$ with the (cycle-)pushforward homomorphism $$\rho_*\colon \Div(Y)_\FF \rightarrow \Div(X)_\FF.$$ In other words, a $\rmb$-$\FF$-divisor $\bfD$ on $X$ is a collection of $\FF$-divisors $\bfD_Y$ on birational models of $X$ that are compatible under pushforward; the divisors $\bfD_Y$ are called the traces of $\bfD$ on the birational models $Y$. A $\rmb$-divisor on $X$ is naturally a $\rmb$-divisor on any birational model of $X$.

The closure of an $\FF$-Cartier $\FF$-divisor $D$ on $X$ is the $\rmb$-$\FF$-divisor $\overline D$ with trace $\overline D_Y=\rho^* D$ for any proper birational morphism $\rho\colon Y\rightarrow X$. A $\rmb$-$\FF$-divisor $\bfD$ over $X$ is $\FF$-Cartier if $\bfD = \overline D$ where  $D$ is an $\FF$-Cartier $\FF$-divisor on a birational model $Y$ over $X$; in this situation, we say $\bfD$ descends to $Y$. 

Let $\bfD$ be a $\FF$-Cartier $\rmb$-$\FF$-divisor that descends to a birational model $Y$ over $X$. We say
\begin{itemize}[leftmargin=*]
\item $\bfD$ is nef/$S$ if  $\bfD_Y$ is nef/$S$;
\item $\bfD$ is \emph{effective}/$S$ (resp.~\emph{num-effective}/$S$, resp.~\emph{num-semiample}/$S$) if $\bfD_Y$ is effective/$S$ (resp.~num-effective/$S$, resp.~num-semiample/$S$).
\end{itemize}
Note that the definitions do not depend on the choice of the birational model $Y$.
\end{defn}
 
 \noindent {\bf Convention.} If $S$ is a point or is clear from the context then we often omit "/$S$" and "$S$" in the notations above.

\subsection{Generalized polarized pairs}

\begin{defn}\label{def: gp}
A \emph{generalized polarized pair} over a scheme $S$, denoted by $(X/S,B+\bfM)$,  consists of 
\begin{itemize}[leftmargin=*]
\item a normal variety $X$ equipped with a projective morphism $X\rightarrow S$,
\item an effective $\RR$-divisor $B$ on  $X$, and
\item an $\RR$-Cartier $\rmb$-$\RR$-divisor $\bfM$ over $X$ that is nef/$S$,
\end{itemize}
such that $K_X+B+\bfM_X$ is $\RR$-Cartier.  
	\end{defn}
  We obtain a pair in the usual sense when $\bfM=0$. Vice versa, from a pair $(X/S, B)$, projective over a scheme $S$, together with a nef/$S$ $\RR$-Cartier $\RR$-divisor $M$, one obtains a generalized polarized pair $(X/S, B+\overline M)$.
  
  As before, if $S$ is a point then we omit it from the notation; in this case, $X$ is projective. Concerning the main results of the paper, we need to assume that $X$ is projective, and $S$ is tacitly taken to be a point. For example, the projectivity of $X$ is needed to apply the Riemann--Roch theorem in Section~\ref{sec: dim2}.

 One can define the \emph{generalized log discrepancy}  of a prime divisor $E$ over $X$ with respect to $(X/S, B+\bfM)$. Suppose $E$ is a divisor on a normal variety $\tilde X$ with a proper birational morphism $f\colon \tilde X \rightarrow X$. Write $K_{\tilde X}+B_{\tilde X}+\bfM_{\tilde X}=f^*(K_{X}+B+\bfM_X)$. Then the generalized log discrepancy of $E$ is (see \cite[Definition 4.1]{BZ16})
	\[
	a_E(X/S, B+\bfM)=1-{\rm mult}_EB_{\tilde X}.
	\]

We say that $(X,B+\bfM)$ is generalized log canonical (generalized lc, for short), resp.~generalized Kawamata log terminal (generalized klt, for short)  if the generalized log discrepancy of any prime divisor over $X$ is $\geq 0$, resp.~$>0$. For a prime divisor $E$ over $X$ with $a_E(X/S, B+\bfM)\leq 0$ we call its image $f(E)\subset X$  a \emph{generalized nonklt center}; the \emph{generalized nonklt locus} is the union of all generalized $\nonklt$ centers, denoted by $\Nklt(X,B+\bfM)$. A generalized lc pair $(X/S, B+\bfM)$ is \emph{generalized dlt} if for the generic point $\eta$ of any generalized nonklt center of $(X/S, B+\bfM)$, $(X, B)$ is log smooth\footnote{A pair $(X, B)$ is log smooth if $X$ is smooth and $B$ is a simple normal crossing divisor on $X$.}\,  near $\eta$ and $\bfM = \overline{\bfM}_X$ over a neighborhood of $\eta$. 

For any generalized lc pair $(X/S, B+\bfM)$, one can construct a \emph{$\QQ$-factorial generalized dlt modification}; see \cite[Lemma 4.5]{BZ16} and \cite[Proposition~3.9]{HL18}. Namely, there is a (projective) birational morphism $h\colon Y\rightarrow X$ and a $\QQ$-factorial generalized dlt pair $(Y/S, B_Y+\bfM)$ such that $h^*(K_X+B+\bfM_X) = K_Y +B_Y +\bfM_Y$, and $a_E(X/S, B+\bfM) = 0$ for any $h$-exceptional divisor $E$.

The minimal model program (MMP) can be defined for generalized polarized pairs (\cite{BZ16,HL18}). As part of the definition, one keeps the $\rmb$-$\RR$-divisor $\bfM$ unchanged in running an MMP.

Recall that a \emph{contraction morphism} is a surjective projective morphism between normal varieties with connected fibres; a \emph{birational contraction} is a birational map $\phi\colon X\dashrightarrow X'$ such that $\phi^{-1}$ does not contract any divisors.
\begin{defn}
Let $(X/S, B+\bfM)$ a generalized lc pair. A \emph{minimal model} of $(X/S, B+\bfM)$ is a $(K_X + B+\bfM_X)$-negative birational contraction $\phi\colon X \dashrightarrow X'$ over $S$ such that $(X'/S, \phi_* B + \bfM)$ is a generalized lc pair and $K_{X'} +\phi_*B+ \bfM_{X'})$ is nef over $S$. The minimal model is \emph{good} (resp.~\emph{numerically good}) if $K_{X'} +\phi_*B+ \bfM_{X'}$ is semiample (resp.~\emph{num-semiample}) over $S$.
\end{defn}

For surfaces we have the following observation.
\begin{lem}\label{lem: glc to lc}
Let $(X, B+\bfM)$ be a generalized lc (resp.~generalized klt) surface. Then  $(X, B)$ is an lc (resp.~klt) surface and $\bfM_X$ is a nef $\Rr$-Cartier $\RR$-divisor. 
\end{lem}
\begin{proof}
Let $f\colon \tilde X\rightarrow X$ be a  log resolution, to which $\bfM$ descends, so $\bfM=\overline{\bfM}_{\tilde X}$ and $\bfM_{\tilde X}$ is nef. Write $f^*(K_X+B+\bfM_X)= K_{\tilde X} + B_{\tilde X}+\bfM_{\tilde X}$. Since $(X, B+\bfM)$ is generalized lc (resp.~generalized klt), the coefficients of $B_{\tilde X}$ are at most 1 (resp.~less than 1). 

By the negativity lemma, we know that $f^*\bfM_{X} - \bfM_{\tilde X}\geq 0$, where $f^*$ is the numerical pull-back of $\RR$-divisors in the sense of Mumford. If we write $f^*(K_X+B)=K_{\tilde X}+B_{\tilde X}'$, then we have $B_{\tilde X}'\leq B_{\tilde X}$. It follows that $(X, B)$ is numerically lc (resp.~numerically klt). By the classification of numerically lc surface singularities, $(X, B)$ has lc (resp.~klt) singularities  (\cite[Section~4.1]{KM98}). 

Therefore, $\bfM_X=(K_X+B+\bfM_X)-(K_X+B)$ is $\Rr$-Cartier. The nefness of $\bfM_X$ follows from that of $\bfM_{X'}$, since $\bfM_X = f_*\bfM_{\tilde X}$ and we are on surfaces.
\end{proof}	

\begin{rmk}
Lemma~\ref{lem: glc to lc} does not generalize to higher dimensions. For example,  let $f\colon \tilde X\rightarrow X$ be a flipping contraction between projective normal varieties so that $\tilde X$ has klt singularities, $-K_{\tilde X}$ is $f$-ample, and $K_X$ is not $\QQ$-Cartier. Since $-K_{\tilde X}$ is $f$-ample, one can find a nef $\Qq$-Cartier $\QQ$-divisor $M_{\tilde X}$ such that $K_{\tilde X}+ M_{\tilde X} \sim_{\Qq, X} 0$. Let $\bfM=\overline{M}_{\tilde X}$. Then  $K_X+\bfM_X = f_*({K_{\tilde X} + M_{\tilde X}})$ is $\QQ$-Cartier and, by the negativity lemma, $ f^*(K_X+\bfM_X) ={K_{\tilde X} + M_{\tilde X}}$. It follows that $(X,\bfM)$ with $B=0$ is a generalized klt pair.  On the other hand,  neither $K_X$ nor $\bfM_X$ is $\Qq$-Cartier.
\end{rmk}

\medskip

\noindent {\bf Convention.} We often abuse the language to describe a generalized polarized pair $(X/S, B+\bfM)$ and its generalized log canonical divisor $K_X+B+\bfM_X$ in the same way. For example, we sometimes say $(X/S, B+\bfM)$ is numerically trivial when $K_X+B+\bfM_X$ is so; the divisor $K_X+B+\bfM_X$ is said to be generalized log canonical if $(X/S, B+\bfM)$ is so.

\subsection{Decomposition of generalized log canonical divisors with real coefficients}

Using Shokurov type polytopes, one can write nef (resp.~antinef) generalized log canonical divisors with $\RR$-coefficients as $\RR_{>0}$-linear combination of nef (resp.~antinef) generalized log canonical divisors with $\QQ$-coefficients (\cite{HL18, HLiu19}). These decompositions are used to reduce problems about generalized log canonical divisors with $\RR$-coefficients to those with $\QQ$-coefficients. We observe that, in order to have such a decomposition, the usual condition that the underlying variety is $\QQ$-factorial dlt as required in \cite{HL18, HLiu19}, can be dropped:

\begin{prop}\label{prop: RtoQ}
Let $(X/S,B+\bfM)$ be a generalized log canonical divisor over a scheme $S$ such that $\bfM =\sum_{1\leq j\leq m} \mu_j^0\bfM_j$ is an $\Rr_{>0}$-linear combination of nef/$S$ Cartier $\rmb$-divisors $\bfM_j$. If $K_X+B+\bfM_X$ is nef/$S$ (resp.~antinef/$S$), then there exist finitely many real numbers $c_{\alpha}>0$ with $\sum_{\alpha} c_{\alpha}=1$, $\QQ$-divisors $B^{\alpha}$ and nef/$S$ $\QQ$-Cartier $\rmb$-$\QQ$-divisors $\bfM^{\alpha}$ on $X$, such that
\begin{itemize}
\item[(i)] $(X/S,B^{\alpha}+\bfM^{\alpha})$ are generalized lc pairs with $\Nklt(X/S,B^{\alpha}+\bfM^{\alpha}) = \Nklt(X/S, B+\bfM)$;
\item[(ii)] $K_X+B= \sum_{\alpha} c_{\alpha} (K_X+B^{\alpha})$ and $\bfM=\sum_{\alpha} c_{\alpha} \bfM^{\alpha}$, and
\item[(iii)]$K_X+B^{\alpha} +\bfM^{\alpha}_{X}$ are nef/$S$ (resp.~antinef/$S$) $\Qq$-Cartier $\QQ$-divisors.
\end{itemize}
\end{prop}
\begin{proof}
Let $f\colon \tilde X\rightarrow X$ be a $\QQ$-factorial generalized dlt modification of $(X,B+\bfM)$, and let $\tilde B$ be the boundary divisor on $\tilde X$ such that $K_{\tilde X} +\tilde B + \bfM_{\tilde X} =f^*(K_X+B+\bfM_X)$. Then $\mult_{E} \tilde B =1$ for any $f$-exceptional divisor $E$. Let $\tilde B=\sum_{1\leq i \leq l} b_i^0 \tilde B_i + \sum_{k} E_k$ be the irreducible decomposition of $\tilde B$, where the $ \tilde B_i$ are the strict transforms of the components of $B$ and the $E_k$ are the $f$-exceptional divisors. By \cite[Proposition 3.17]{HL18} and \cite[Proposition~2.7]{HLiu19}, there exists a rational polytope $P$ in the $(l+m)$-dimensional $\RR$-vector space $V$ such that $\rmv=(b_1,\dots, b_l, \mu_1,\dots, \mu_m)\in V$ lies in $P$ if and only if  $(\tilde X, \tilde B_{\rmv} + \bfM_{\rmv})$ is a nef/$S$ (resp.~antinef/$S$) generalized lc pair, where $B_{\rmv}= \sum_i b_i \tilde B_i + \sum_k E_k$ and  $\bfM_{\rmv}=\sum_j \mu_j \bfM_j$. Note that $\tilde B =\tilde B_{\rmv^0}$ and $\tilde \bfM = \bfM_{\rmv^0}$ for $\rmv^0 = (b_1^0,\dots, b_l^0, \mu_1^0,\dots, \mu_m^0)$, and thus $\rmv^0\in P$. 

Note also that $$K_{\tilde X} + \tilde B_{\rmv^0}  + \bfM_{\rmv^0, \tilde X} = K_{\tilde X} +\tilde B + \bfM_{\tilde X} \sim_{\QQ, f} 0.$$ By the argument of \cite[Lemma~3.1]{HLiu19}, there is an affine subspace $A$ of $V$, defined over $\QQ$ and passing through $\rmv^0$, such that  $K_{\tilde X}+\tilde B_{\rmv} + \bfM_{\rmv}\sim_{\QQ, f} 0$ for any $\rmv\in A$. 

Now it is easy to see that there are rational vectors $\rmv^1, \dots, \rmv^p\in P\cap A$, which can be sufficiently close to $\rmv^0$, and positive real numbers $c_1,\dots, c_p$ such that 
\[
\sum_{1\leq \alpha\leq p} c_{\alpha} =1, \tilde B = \sum_{1\leq \alpha\leq p} c_{\alpha} \tilde B_{\rmv^{\alpha}}, \bfM = \sum_{1\leq \alpha\leq p} c_{\alpha} \bfM_{\rmv^{\alpha}}, 
\]
Moreover, for the rational coordinates $\rmv^0_i$ of  $\rmv^0$ we can make sure that $\rmv^\alpha_i = \rmv^0_i$ for each $1\leq \alpha\leq p$. Let $B^{\alpha} = f_*\tilde B_{\rmv^{\alpha}}$ and $\bfM^{\alpha}  = \bfM_{\rmv^{\alpha}}$. By construction, the generalized lc divisors $K_X+B^{\alpha} +\bfM^{\alpha}_{X}$ are $\QQ$-Cartier, and the data $c_{\alpha}, B^{\alpha}, \bfM^{\alpha}$ satisfy (i)-(iii) of the proposition.
\end{proof}


\section{Numerical nonvanishing in dimension two}\label{sec: dim2}

\subsection{Numerical nonvanishing for surfaces}

In this subsection, we prove Conjecture~\ref{conj: nonvanishing} in dimension two:

\begin{thm}\label{thm: nv}
Let $(X, B+\bfM)$ be a projective log canonical surface such that $\bfM$ is an $\RR_{> 0}$-linear combination of nef Cartier $\rmb$-divisors. If $K_X+B+\bfM_X$ is pseudo-effective, then there exists an effective $\RR$-Cartier $\RR$-divisor $D$ such that $K_X+B+\bfM_X\equiv D$. Moreover, if $B$ and $\bfM_X$ are $\QQ$-divisors then we can additionally require $D$ to be a $\QQ$-divisor.
\end{thm}

We will use the following lemmas.
\begin{lem}\label{lem: Hodge}
Let $X$ be a normal projective surface. Let $A$ and $B$ two nef $\Rr$-Cartier $\RR$-divisors satisfying $A\cdot B=0$. Then numerical classes of $A$ and $B$ in $N_1(X)_\RR=N^1(X)_\RR$ are proportional to each other.
\end{lem}
\begin{proof}
By replacing $X$ with a resolution, and $A$ and $B$ with their pull-backs, we may assume that $X$ is smooth. We may also assume that neither $A$ nor $B$ are numerically trivial. Let $H$ be an ample divisor on $X$. By the Hodge index theorem, $H\cdot A>0$ and $H\cdot B>0$, so there is a positive real number $a$ such that $H\cdot A=aH\cdot B$. It follows that $H\cdot(A-aB) =0$. Since $(A-aB)^2=A^2-2aA\cdot B + B^2=0$, by the Hodge index theorem again, $A\equiv aB$. 
\end{proof}

\begin{lem}\label{lem: rational}
Let $X$ be a projective lc surface. If $\chi(\mo_X)\leq 0$ then $X$ has at most rational singularities and hence is $\QQ$-factorial.
\end{lem}
\begin{proof}
Let $f\colon\tilde X\rightarrow X$ be a resolution of singularities. Suppose there are $k>0$ elliptic singularities on $X$. Since $\chi(\mo_X)\leq 0$ by assumption, we have
$$\chi(\mo_{\tilde X}) = \chi(\mo_X) - h^0(X, R^1f_*\mo_{\tilde X}) = \chi(\mo_X) -k<0.$$ 
Therefore, $\tilde X$ is birational to a ruled surface over a curve of genus at least $2$ and it follows that $X$ cannot have elliptic singularities, which is a contradiction to the assumption.

For the second half of the statement, we note that rational normal surface singularities are $\QQ$-factorial by \cite[Lemma~2.12]{Na04}.
\end{proof}

\begin{proof}[Proof of Theorem~\ref{thm: nv}]
First we make some preparatory reductions. After running a $(K_X+B+\bfM_X)$-MMP, which is automatically a $(K_X+B)$-MMP, we may assume that $K_X+B+\bfM_X$ is nef. By Proposition~\ref{prop: RtoQ}, we can assume that $B$ has $\QQ$-coefficients and $\bfM$ is a $\QQ_{>0}$-linear combination of nef Cartier $\rmb$-divisors. Finally, we can assume that neither $\bfM_X$ nor $K_X+B+\bfM_X$ is numerically trivial.

Let $m\geq 2$ be a positive integer such that $m(K_X+B+\bfM_X)$ is Cartier. The usual Riemann--Roch formula holds for Cartier divisors on normal projective surfaces (\cite{Bl95}):
\begin{multline*}\label{eq: RR}
h^0(X, m(K_X+B+\bfM_X))-h^1(X, m(K_X+B+\bfM_X)) + h^2(X, m(K_X+B+\bfM_X))  \\
= \chi(\mo_X) + \frac{1}{2}\left(m(m-1)(K_X+B+\bfM_X)^2 + m(K_X+B+\bfM_X)\cdot (B+\bfM_X)\right)
\end{multline*}
By the Serre duality we have $$h^2(X, m(K_X+B+\bfM_X))  = h^0(X, -B-\bfM_X-(m-1)(K_X+B+\bfM_X))=0.$$ It follows that $h^0(X, m(K_X+B+\bfM_X))>0$ for sufficiently large and divisible $m$ unless $\chi(\mo_X)\leq 0$ and $(K_X+B+\bfM_X)^2 = (K_X+B+\bfM_X)  \cdot (B+\bfM_X)=0$. Since $K_X+B+\bfM_X$ and $\bfM_X$ are nef, we obtain $$(K_X+B+\bfM_X) \cdot K_X = (K_X+B+\bfM_X) \cdot B =(K_X+B+\bfM_X)\cdot \bfM_X = 0.$$ By Lemma~\ref{lem: Hodge}, the divisor $K_X+B+\bfM_X$, and hence $K_X+B$, is numerically proportional to $ \bfM_X$. It follows that either $K_X+B$ or $-(K_X+B)$ is nef.  If $K_X+B$ is nef and not numerically trivial, then there is a postive rational number $t$ such that $K_X+B+\bfM_X \equiv t(K_X+B)$ which is num-effective by the nonvanishing for nef log canonical divisors in dimension two.

From now on, we can assume that $\chi(\mo_X)\leq 0$, $K_X+B \equiv -a \bfM_X$ from some $a\in (0,1)\cap\QQ$, and  $ \bfM_X^2 =  \bfM_X\cdot  B =0$. 

Let $f\colon\tilde X\rightarrow X$ be the minimal resolution and write $f^*(K_{ X}+{ B}+\bfM_X) = K_{\tilde X} + B_{\tilde X} +\bfM_{\tilde X}$.  Since $\chi(\mo_X)\leq 0$, $X$ is $\QQ$-factorial by Lemma~\ref{lem: rational}. Hence it suffices to show that $K_{\tilde X} + B_{\tilde X} +\bfM_{\tilde X}$ is num-effective. We can thus replace $(X, B+\bfM)$ with $(\tilde X, B_{\tilde X}+\bfM)$, and assume that $X$ is smooth.

Now define the pseudo-effective threshold $$t_0:=\inf\{t\geq 0 \mid \, K_X+tB+ \bfM_X\text{ is pseudo-effective}\}.$$ We divide the discussion according to whether $t_0=0$ or $t_0>0$.

\medskip

\noindent{\bf Case 1:} $t_0=0$, so $K_X+\bfM_{X}$ is pseudo-effective. It suffices to show that $K_X+\bfM_X$ is num-effective. After running a $(K_X+\bfM_X)$-MMP, which contracts only $(-1)$-curves, we can assume that $K_X+\bfM_X$ is nef. As before, using the Riemann--Roch theorem, we can reduce to the case where $\chi(\mo_X)\leq 0$, $K_X\equiv-b\bfM_X$ for some $b\in(0,1)\cap\QQ$, and $K_X^2  = \bfM_X^2=0$. Then $X$ is necessarily a (minimal) ruled surface over an elliptic curve. In this case,  $-K_X$ is num-effective (cf.~\cite[Example 1.1]{Sh00}). Therefore $K_X+\bfM_X\equiv -(\frac{1}{b}-1)K_X$ is also num-effective.


\medskip

\noindent{\bf Case 2:} $t_0>0$. In this case, we run a $(K_X+t_1B + \bfM_X)$-MMP for $0<t_1<t_0$ and $t_0-t_1$ sufficiently small, and it  necessarily terminates to a Mori fibre space $h\colon Y\rightarrow Z$ of $K_Y+t_1B_Y +\bfM_Y$, where $B_Y$ is the strict transform of $B$ on $Y$. According to \cite[Lemma 3.18]{HL18}, the $(K_X+t_1B + \bfM_X)$-MMP is $(K_X+t_0 B+\bfM_X)$-trivial. 

Let $F$ be the general fiber of $Y\to Z$. We have
$$(K_Y+t_0 B_Y+\bfM_Y)|_{F}\coloneqq K_F+t_0 B_F+M_F\equiv 0.$$
By taking the degree, we know that $t_0\in\Qq$. 

If $\dim Z =0$ then $K_Y+t_0 B_Y+\bfM_Y\sim_{\QQ} 0$. If $\dim Z= 1$, then $K_{Y}+t_0 B_Y+\bfM_Y\sim_{\QQ}h^{*}L$ for some divisor $L$ of nonnegative degree on $Z$. It follows that $K_{Y}+t_0 B_Y+\bfM_Y$ and hence $K_X+t_0 B+\bfM_X$ is num-effective.
\end{proof}

\begin{proof}[Proof of Corollary \ref{cor: antinef}]
By Proposition~\ref{prop: RtoQ} we can assume that $B$ is a $\Qq$-divisor. Setting $M:=-2(K_X+B)$, there is a nef effective $\QQ$-divisor $D$ such that $-(K_X+B)=K_X+B+M\equiv D$ by Theorem~\ref{thm: nv}.

If $-(K_X+B)\equiv 0$ then $K_X+B\sim_{\QQ} 0 $ by the abundance of $K_X+B$; see, for example, \cite{Fu12}.  To the other extreme, if $-(K_X+B)$ is big, then it is necessarily $\QQ$-linearly equivalent to an effective $\QQ$-divisor. Also, if $h^1(X, \mo_X)=0$ then the numerical equivalence for $\QQ$-Cartier $\QQ$-divisors on $X$ is the same as the $\QQ$-linear equivalence, so $-(K_X+B)\sim_{\QQ} D$. 

In the following, we assume that $-(K_X+B)\equiv D$ is nef but is neither numerically trivial nor big, and $h^1(X, \mo_X) >0$. Let $h\colon \tilde X\rightarrow X$ be the minimal resolution and write $h^*(K_X+B) = K_{\tilde X}+B_{\tilde X}$. Then $h^1(\tilde X, \mo_{\tilde X}) >0$ and it suffices to show that  $-(K_{\tilde X} + B_{\tilde X})\sim_{\QQ} \tilde D$ for some effective $\QQ$-divisor $\tilde D$. By replacing $(X, B)$ with $(\tilde X, B_{\tilde X})$ we can thus assume that $X$ is a smooth projective surface. The Albanese map induces  a $\PP^1$-fibration $f\colon X\rightarrow C$ onto a curve of positive genus. Moreover, by contracting $(-1)$-curves $E$ with $-(K_X+B)\cdot E=0$ we can assume that $-(K_X+B)$ intersects any $(-1)$-curve positively.

Let $D=\sum d_i D_i$ be the decomposition into irreducible components.  Since $D$ is nef but not big, we have $D\cdot D_i\geq 0$ for each $i$ and $ D^2= \sum_i (1/d_i)D\cdot D_i =0$, and hence $D\cdot D_i =0$ for each $i$. It follows that
\begin{equation}\label{eq: negative H1}
D_i^2  = \frac{1}{d_i}\left(D\cdot D_i - \sum_{j\neq i} d_jD_j \cdot D_i\right) \leq  \frac{1}{d_i} \left(D\cdot D_i\right) = 0.
\end{equation}

If $D\cdot F=0$, where $F$ is a general fibre of $f\colon X\rightarrow C$, then $D = f^* A$ for some effective $\QQ$-divisor $A$ on $C$ of positive degree. It follows that $-(K_X+B)$ is $\QQ$-linearly equivalent to the pull-back of an effective divisor $A'\equiv A$ on $C$.


From now on, we assume furthermore that $D\cdot F >0$. Consider any irreducible curve $H$, horizontal with respect to $f$, such that $H^2\leq 0$. (For example, any horizontal component of $D$ is such a curve.) Denote $b:=\mult_{H}B$. Then $b\leq 1$ and
\begin{equation*}
0\leq -(K_X+B) \cdot H  = -(K_X+H) \cdot H + (1-b) H^2  - (B-bH)\cdot H\leq 0
\end{equation*}
It follows that 
\begin{equation}\label{eq: negative H}
-(K_X+H) \cdot H = (1-b) H^2  = (B-bH)\cdot H = 0,
\end{equation}
thus $p_a(H) =1$ and $H$ does not intersect any other component of $B$. Since $f|_{H}\colon H\rightarrow C$ is dominant and $g(C)\geq 1$, we infer that $H$ and $C$ are both smooth elliptic curves.

If there is a horizontal curve $H$ such that $H^2<0$ then \eqref{eq: negative H} implies that $b=1$. Since $-(K_X+B)\cdot F = D\cdot F>0$ for a fibre $F$ of $f\colon X\rightarrow B$, we obtain
\[
H \cdot F\leq B\cdot F < -K_X\cdot F =2,
\]
thus $H\cdot F=1$. In other words, $H$ is a section of $f\colon X\rightarrow C$ with negative self-intersection. Then there is a birational morphism $\pi\colon X\rightarrow \bar X$ contracting exactly $H$ to a simple elliptic singularity. Let $B_{\bar X} = \pi_* B$. Then $(\bar X,B_{\bar X})$ is log canonical with $\pi^*(K_{\bar X}+B_{\bar X}) = K_X+B$, which is antinef. It holds $h^1(\bar X, \mo_{\bar X})=0$, so  $-(K_{\bar X}+B_{\bar X})$ is $\QQ$-linearly equivalent to an effective $\QQ$-Cartier $\QQ$-divisor $\bar D$ on $\bar X$ as before. It follows that $$-(K_X+B)\sim_{\QQ} \pi^* \bar D\geq 0.$$ 

Now we can assume that any horizontal curve $H$ of $f\colon X\rightarrow C$ has $H^2 \geq 0$. Let $\varphi\colon X\rightarrow X_m$ be the birational morphism to a smooth model $X_m$ without $(-1)$-curves. We will show that $\varphi$ is an isomorphism. Note that the natural morphism $f_m\colon X_m\rightarrow C$ is a $\PP^1$-bundle, all of whose sections have nonnegative self-intersection. Let $B_{m}=\varphi_* B$ and $D_m=\varphi_* D$. Then $-(K_{X_m} + B_m)\equiv D_m$ is nef. The class of $-K_{X_m}$ spans one boundary ray $R$ of $\overline{NE}(X)$ (see \cite[1.5.A]{L04} or \cite{Sh00}) and the classes of $B_m$ and $D_m$ lie necessarily on $R$. We have thus $D_m^2 = K_{X_m}^2 = 0$. Since $D^2= D_m^2 =0$, the morphism $\varphi\colon X\rightarrow X_m$ cannot blow up any point on $D_m$. Also, the fact that $D\cdot E>0$ for any $(-1)$-curve $E$ on $X$ implies that $\varphi$ does not blow up points outside $D_m$ either. This implies that $X=X_m$. We conclude by Lemma~\ref{lem: elliptic ruled surface} below.
 \end{proof}

\begin{lem}\label{lem: elliptic ruled surface}
Let $f \colon X\rightarrow C$ be a ruled surface over an elliptic curve $C$ such that any section of $f$ has nonnegative self-intersection. Let $B$ be a boundary $\RR$-divisor on $X$ such that $(X, B)$ is log canonical. If $-(K_X+B)$ is nef then $-(K_X+B)\sim_{\RR} D$ for some effective $\RR$-Cartier $\RR$-divisor $D$.
\end{lem}
\begin{proof}
By Proposition~\ref{prop: RtoQ} we can assume that $B$ is a $\QQ$-divisor. Let $C_0$ be a section of $f\colon X\rightarrow C$ such that $-e:=C_0^2$ attains the minimum. Then $e=0$ or $-1$; see \cite[V.2]{Har77}. As we have already seen in the proof of Corollary~\ref{cor: antinef}, the class of $K_X$ generates one boundary ray, say $R$, of $\overline{NE}(X)$.  Moreover, the linear system $|-2K_X|$ is nonempty; see \cite[Corollary 2.2]{Sh00}. Since $-(K_X+B)$ is nef, the class of $B$ necessarily lies on $R$. Note that, if $B\sim_\QQ -aK_X$ for some $a\in \QQ_{\geq 0}$, then $a\leq 1$ and $-(K_X+B)\sim_\QQ (1-a)K_X$ is $\QQ$-linearly equivalent to an effective $\QQ$-divisor.

Let $\me=f_*\mo_X(C_0)$, which is a rank two vector bundle of degree $-e$. If $\me$ is indecomposable then $B\sim_{\QQ} -aK_X$ is for some $a\geq 0$ (\cite[Examples 1.1 and 2.1]{Sh00}). If $\me$ is decomposable then $\me=\mo_X\oplus \mo_X(\fe)$ for some $\fe\in\Pic^0(C)$ and $-K_X\sim C_0 + C_0'$ where $C_0'\in |C_0 - \fe|$ is another section of $f$. In case $\fe$ is a torsion in $\Pic^0(C)$, $|mK_X|$ is base point free for sufficiently divisible $m$, inducing an elliptic quasi-bundle $X\rightarrow \PP^1$; it is clear now that $B\sim_{\QQ} -aK_X$ for some $a\in\QQ_{\geq 0}$. If $\fe$ is not a torsion in $\Pic^0(X)$ then $C_0$ and $C_0'$ are the only curves whose numerical classes lie on $R$ (cf.~\cite{Ho80}). In this case, $B= aC_0 + a' C_0'$ with $0\leq a, a' \leq 1$. Therefore, $-(K_X+B)\sim_{\QQ} (1-a) C_0 + (1-a') C_0'\geq 0$.
\end{proof}



\begin{proof}[Proof of Corollary \ref{cor:serranosurface}]
The argument, already used in \cite[Remark-Corollary 2.6]{Sh00}, is by now quite standard. By the Cone Theorem, $K_X+B+tL$ is strictly nef for any $t>3$, thus there exists an effective $\Rr$-divisor $D$ such that $K_X+B+tL\equiv D$ by Theorem~\ref{thm: nv}. Since $(K_X+B+tL)^2=(K_X+B+tL)\cdot D>0$, $K_X+B+tL$ is big. Now, the corollary follows from 
the Nakai criterion for the ampleness of $\Rr$-Cartier $\RR$-divisors; see  \cite{CP90} and \cite[Theorem 2.3.18]{L04}.
\end{proof}

In view of Corollary~\ref{cor:serranosurface} we ask the following question, which is meant to generalize a conjecture of Serrano \cite{Se95}; see also \cite[Conjecture~0.1]{CCP08}. 
\begin{ques}
	Let $(X,B)$ be a projective log canonical pair of dimension $n$ and let $L$ be a strictly nef Cartier divisor on $X$. Is $K_X+B+tL$ ample for every real number $t>n+1$?
\end{ques}

\subsection{Effective numerical nonvanishing for surfaces}
In this subsection we prove two effective  numerical nonvanishing results, Theorems~\ref{thm: eff nv klt} and \ref{thm: dcc}, for generalized log canonical surfaces. 

\begin{thm}\label{thm: eff nv klt}
Let $(X, B)$ be a projective log canonical surface with $\QQ$-coefficients and $M$ a nef $\QQ$-Cartier $\QQ$-divisor. If $r(K_X+B+M)$ is nef and Cartier for some positive integer $r$, then $Nr(K_X+ B+M)$ is numerically equivalent to an effective Cartier divisor for any integer $N\geq 3$.
\end{thm}
We first deal with elliptic quasi-bundles.
\begin{definition}[\cite{Se91}]
A projective morphism $f\colon X\rightarrow C$ from a smooth projective surface onto a smooth curve is called a \emph{quasi-bundle} if all of the fibres are connected, the smooth fibres are pairwise isomorphic and the only singular fibres are multiples of smooth curves.
\end{definition}

\begin{prop}\label{prop: ell num eff}
Let $f\colon X\rightarrow C$ be an elliptic quasi-bundle from a smooth projective surface onto a curve. Let $M$ be a nef Cartier divisor such that $K_X+M$ is nef and $M\cdot F=0$ where $F$ is a general fibre of $f$. Assume that one of the following conditions is satisfied:
\begin{enumerate}[leftmargin=*]
\item[(i)] $f$ is an elliptic bundle, that is, it does not have singular fibres at all;
\item[(ii)] $g(C)\geq 1$;
\item[(iii)] $q(X)=g(C)+1$.
\end{enumerate}
Then $N(K_X+M)$ is numerically equivalent to an effective Cartier divisor for integers $N \geq 3$.
\end{prop}
\begin{rmk}
It is indeed possible that $2(K_X+M)$ is not numerically equivalent to any effective Cartier divisor; see \cite[Remark~4.5]{Se91}.
\end{rmk}
\begin{proof}[Proof of Proposition~\ref{prop: ell num eff}] 

Case (i) follows by \cite[Theorem~4.3]{Se91}.

For the remaining two cases, let $m_1F_1,\dots, m_kF_k$ be the multiple fibres. Let $m=\lcm(m_1,\dots, m_k)$.  By \cite[Theorems 4.1 and 4.4]{Se91}, we have the following description of  $N_1(X)$:
\begin{itemize}[leftmargin=*]
\item The integral numerical classes that is proportional to $[F]\in N_1(X)$ are exactly the integral multiples of $(1/dm)[F]$ for some fixed $d\leq 3$. Thus $M\equiv (a/dm)F$ for some nonnegative integer $a$.
\item  The semigroup of the numerical classes of curves that are proportional to $[F]$ are generated by 
$[(1/m_1) F], \ldots, [(1/m_k) F].$
\end{itemize}
Let $G\subset \ZZ_{\geq 0}$ be the numerical semi-group generated by $m/m_1,\dots, m/m_k$. Its Frobenius number $\Fr(G)$ is defined to be the largest integer that does not lie in $G$. Since
\[
K_X+M\equiv\left(2g(C)-2+\sum_{1\leq i\leq k} \left(1-\frac{1}{m_i}\right) +\frac{a}{dm}\right)F,
\]
it suffices to show that $dK(G)+ a> \Fr(G),$ where $K(G):=m(2g(C)-2+\sum_{1\leq i\leq k}(1-1/m_i) )$,

Define the common factor graph $\Gamma$ of the natural numbers $m_1,\dots, m_k$ so that there are $k$ vertices $v_i$ corresponding to the multiplicities $m_i$ respectively, and $v_i$ and $v_j$ are connected by an edge if $\gcd(m_i, m_j)>1$.  Then relabel  $m_1,\dots, m_k$ in such a way that the subgraphs containing $\{v_1, \dots, v_{k_1}\}$, $\{v_{k_1+1}, \dots, v_{k_1+k_2}\}$, $\dots, \{v_{k_1+\dots+k_{s-1}+1}, \dots, v_k\}$ respectively form the connected components of the graph $\Gamma$. Obviously, for an index $i\notin\{1, k_1+1,k_1+k_2+1, \dots,k_1+\cdots +k_{s-1}+1 \}$, one has $\gcd(\lcm(m_1, \dots, m_{i-1}), m_i)>1.$

By \cite[Theorem 2]{Br42}, we have
\begin{equation}\label{eq: Br}
\Fr(G) \leq \sum_{2\leq i\leq k}\frac{m}{m_i}\cdot\frac{d_{i-1}}{d_i} -\sum_{1\leq i \leq k} \frac{m}{m_i} = m\sum_{2\leq i\leq k} \left(\frac{d_{i-1}}{m_id_i} - \frac{1}{m_i}\right) - \frac{m}{m_1}.
\end{equation}
where $d_i=\gcd(m/m_1,\dots,m/m_i)$ for $1\leq i\leq k$. We compute for $2\leq i\leq k$
\[
\frac{d_{i-1}}{m_id_i} = \frac{\lcm(m_1,\dots, m_i)}{\lcm(m_1,\dots, m_{i-1})\cdot m_i}=\frac{1}{\gcd(\lcm(m_1,\dots, m_{i-1}), m_i)}.
\]
Denoting $e_i=\gcd(\lcm(m_1,\dots, m_{i-1}), m_i)$ for $i\geq 2$, it follows from \eqref{eq: Br} that
\begin{equation}\label{eq: K-F}
\begin{split}
&\left(dK(G)+a\right)-\Fr(G) \\
 \geq & dm\left(2g(C)-2+\sum_{1\leq i\leq k} \left(1-\frac{1}{m_i}\right)\right) +a - m\sum_{2\leq i\leq k} \left(\frac{1}{e_i} - \frac{1}{m_i}\right) + \frac{m}{m_1}\\
=&dm\left(2g(C)-2 +1-\left(1-\frac{1}{d}\right)\frac{1}{m_1} +\frac{a}{dm}+ \sum_{2\leq i\leq k}\left(1 +\frac{1}{dm_i}-\frac{1}{m_i}-\frac{1}{de_i}\right)\right).
\end{split}
\end{equation}
which is positive as soon as $g(C)\geq 1$. So case (ii) of the lemma is confirmed.

Now we can assume that we are in case (iii) but not in cases (i) and (ii), so $g(C)=0$. In this case we have $d=1$ by \cite[Theorem~4.1]{Se91}, and \eqref{eq: K-F} becomes
\begin{equation}\label{eq: K-F'}
K(G) + a -\Fr(G) \geq m\left( -1 + \frac{a}{m}+\sum_{2\leq i\leq k}\left(1- \frac{1}{e_i}\right)\right).
\end{equation}

If $s\geq 2$, then
\begin{multline*}
 -1 + \sum_{2\leq i\leq k}\left(1- \frac{1}{e_i}\right) \geq -1 + 1 -\frac{1}{e_2} +1 - \frac{1}{e_{k_1+2}}
   = 1-\frac{1}{e_2} - \frac{1}{e_{k_1+2}} \\
   \geq 1-\frac{1}{2} - \frac{1}{3} >0.
 \end{multline*}
 where the second inequality is because $e_i\geq 2$ for each $i\geq 2$, and $e_2$ and $e_{k_1+2}$ are coprime.

 If $s=1$, then $k_1=k$. Since $g(C)=0$, we have in any case $k\geq 2$ by looking at the monodromy of $f$. One see that
$  \sum_{2\leq i\leq k}(1- 1/e_i)\leq 1$ only if
\begin{enumerate}
\item[(i)] $k=3$ and $e_1=e_2=e_3=2$, in which case $m=m_1=m_2=m_3=2$, or
\item[(ii)] $k=2$, in which case $m=m_1=m_2$.
\end{enumerate}
In both cases we have $G=\Zz_{\geq 0}$, so $K_X+M$ is numerically equivalent to an effective Cartier divisor.
\end{proof}

For the proof of Theorem~\ref{thm: eff nv klt}, we need the numerical version of an effective nonvanishing result of Kawamata (\cite{K00}).
\begin{lem}\label{lem: nv kwmt}
Let $(X, B)$ be a projective log canonical surface and $M$ a big and nef $\QQ$-Cartier $\QQ$-divisor on $X$. If $K_X+B+M$ is nef and Cartier (so in particular the sum $B+M$ has integer coefficients), then either $\kappa(K_X+B+M)\geq 0$ or $K_X+B+M\equiv 0$, where $\kappa(K_X+B+M)$ denotes the Kodaira--Iitaka dimension of $K_X+B+M$.
\end{lem}
\begin{proof}
If $(X, B)$ is klt then the assertion follows from \cite[Theorem~3.1]{K00}. 

Now suppose  that $(X, B)$ has nonempty nonklt locus. By taking a dlt modification (\cite[Theorem~4.4.21]{Fu17}) and then the minimal resolution, we can assume that $X$ is smooth, $(X, B)$ is dlt, and $\llcorner B \lrcorner \neq 0$. 

The Riemann--Roch Theorem reads
\begin{multline*}
h^0(X, K_X+B+M)-h^1(X, K_X+B+ M) = \chi(\mo_X) + \frac{1}{2}(K_X+B+M)\cdot (B+M).
\end{multline*}
The right side of the above equation is positive unless
\begin{itemize}
\item[(i)] $\chi(\mo_X)  = (K_X+B+M)\cdot (B+M) = 0$, or
\item[(ii)] $\chi(\mo_X)<0$.
\end{itemize}
In case (i), $K_X+B+M \equiv 0$ by the Hodge index theorem. In case (ii), the Albanese map gives a $\PP^1$-fibration $f\colon X\rightarrow C$ over a curve with $g(C)=1-\chi(\mo_X)\geq 2$. Since $M$ is big and nef, the Kawamata--Viehweg vanishing gives
\[
 0 = H^1(X, K_X+\ulcorner M \urcorner) = H^1(X, K_X+B+M-\llcorner B \lrcorner ),
\]
and hence a surjection
\[
H^0(X, K_X+B+M)\twoheadrightarrow  H^0(\llcorner B \lrcorner, (K_X+B+M)|_{\llcorner B \lrcorner}).
\]
Note that $ \llcorner B \lrcorner$ is a nodal curve. An irreducible component $\llcorner B \lrcorner$ is either vertical with respect to $f$ and hence has arithmetic genus 0, or dominant onto the base curve $C$ and hence has geometric genus at least 2. For any connected subcurve $D$ of  $\llcorner B \lrcorner$, since $K_X+B+M$ is nef and $\deg (K_X+B+M)|_D \geq \deg K_D +\deg \Diff_D(B-D)$, where $\Diff_D(B-D)$ is an effective divisor supported on the smooth locus of $D$, we have 
\[
\deg (K_X+B+M)|_D \geq \max \{0, 2p_a(D) -2\}.
\]
Taking $D$ to be a connected component of  of  $\llcorner B \lrcorner$ one deduces that $$h^0(X, K_X+B+M)\geq h^0(\llcorner B \lrcorner, (K_X+B+M)|_{\llcorner B \lrcorner})\geq h^0(D, (K_X+B+M)|_D) >0,$$ where the last inequality is by the Riemann--Roch theorem for embedded curves (\cite[Theorem~3.1]{BHPV04}).
\end{proof}

\begin{proof}[Proof of Theorem~\ref{thm: eff nv klt}]
We divide the proof into two steps.

\medskip

\noindent{\bf Step 1.} In this step we assume that $K_X+B+M$ is Cartier. By replacing $X$ with its minimal resolution and $K_X+B+M$ its pull-back, we can assume that $X$ is smooth, so $B+M$ is Cartier. After contracting $(-1)$-curves $E$ such that $(K_X+B+M)\cdot E=0$ we can assume that $K_X+B+M$ intersects any ($-1$)-curve (if existing) positively. We also assume $K_X+B+M$ is not numerically trivial, otherwise there is nothing to prove. 

If $K_X+B+2M$ is big, then for $N\geq 2$, $(N-1)(K_X+B+M)+M$ is big and nef. By Lemma~\ref{lem: nv kwmt}, we know that $N(K_X+B+M)$ is numerically equivalent to an effective Cartier divisor. 

In the following we can assume that $K_X+B+2M$ is not big. Then it necessarily holds:
\[
(K_X+B+M)^2=(K_X+B+M)\cdot M=0
\]
By Lemma~\ref{lem: Hodge}, we infer that $K_X+B+M$ and $M$ are numerically proportional to each other. Since $H^2(X, N(K_X+B+M))=0$ for $N\geq 2$, by the Riemann--Roch Theorem,
\begin{multline*}
h^0(X, N(K_X+B+M))-h^1(X, N(K_X+B+ M))\\= \chi(\mo_X) + \frac{1}{2}N(K_X+B+M)\cdot B.
\end{multline*}
As in the proof of Lemma~\ref{lem: nv kwmt}, the right side of the above equation, and hence $h^0(X, N(K_X+B+M))$, is positive, unless 
\begin{itemize}
\item[(i)]  $\chi(\mo_X) = (K_X+B+M)\cdot B=0$, or 
\item[(ii)] $\chi(\mo_X)<0$.
\end{itemize}
In case (i), a smooth minimal model of $X$ is one of the following:
\begin{enumerate}
\item[(ia)] a ruled surface over an elliptic curve,
\item[(ib)] an abelian surface,
\item[(ic)] a bi-elliptic surface,
\item[(id)] a properly elliptic surface.
\end{enumerate}

\medskip

\noindent{Case (ia).} Let $f\colon X\rightarrow C$ be the Albanese fibration, which is a $\Pp^1$-fibration over an elliptic curve $C$. Let $F$ be a general fibre of $f$.  The Riemann--Roch Theorem gives that
\begin{equation}\label{eq: serrano}
h^0(X, K_X+B+M+\alpha) \geq h^1(X, K_X+B+M+\alpha).
\end{equation}
where $\alpha\in \Pic^0(X)$. Note that either $h^1(X, K_X+B+M+\alpha)=0$ for any $\alpha\in \Pic^0(X)$ or $h^1(X, K_X+B+M+\alpha_0)=0$ for some $\alpha_0\in \Pic^0(X)$. By \cite[Proposition 1.5]{Se95} and \eqref{eq: serrano}, $K_X+B+M$ is numerically equivalent to an effective Cartier divisor in both cases.

\medskip

In cases (ib), (ic) and (id), $X$ is minimal: otherwise $K_X$ is an numerically equivalent to an effective divisor containing at least one  $(-1)$-curve $E$ and hence $(K_X+B+M)\cdot K_X\geq (K_X+B+M)\cdot E>0$ by the assumption made in the beginning of the proof. 

\medskip

\noindent{Case (ib).} The divisor $K_X+B+ M$ is numerically equivalent to an effective Cartier divisor by \cite[Lemma~1.1]{Ba98}.

\medskip

\noindent{Case (ic).} In this case $B+M\equiv K_X+B+M$ is nef and numerically proportional to the fibres of an elliptic quasi-bundle $f\colon X\rightarrow C$ satisfying the conditions of Proposition~\ref{prop: ell num eff}. Thus $N(K_X+B+M)\equiv N(B+M)$ is numerically equivalent to an effective Cartier divisor for $N\geq 3$.

\medskip

\noindent {Case (id):} In this case $K_X+B+M$ is numerically proportional to a fibre $F$ of the Iitaka fibration $f\colon X\rightarrow C$. Since the topological Euler characteristic vanishes: $$e(X)=12\chi(\mo_{X})-K_{X}^2=0,$$ one sees that $f$ is an elliptic quasi-bundle and Proposition~\ref{prop: ell num eff} does the job.

\medskip

\noindent Case (ii). In this case the Albanese map gives a $\Pp^1$-fibration $f\colon X\rightarrow C$ onto a curve $C$ with $g(C)=q(X)\geq 2$. Let $F$ be a fibre of $f$.

\medskip

\noindent{\bf Claim.} $(K_X+B+M)\cdot F=0$.

\begin{proof}[Proof of the claim]
By Theorem~\ref{thm: nv} there is an effective Cartier divisor $G$ which is numerically equivalent to $N(K_X+B+M)$ for some positive integer $N$. Let  $G_i$ be the irreducible components of $G$. Since $G$ is nef and $G^2=0$, one sees easily that $G_i^2\leq 0$ and $K_X\cdot G_i= - (B+M)\cdot G_i \leq -G_i^2$. It follows that for any $i$
\begin{equation}\label{eq: pa}
p_a(G_i) =1+ \frac{1}{2}(K_X+G_i)\cdot G_i\leq 1.
\end{equation}
Since $g(C)\geq 2>p_a(G_i)$, every $G_i$ must be vertical with respect to $f$, which means $(K_X+B+M)\cdot F=0$.
\end{proof}
Now that $(K_X+B+M)\cdot F=0$, $f\colon X\rightarrow C$ is a $\Pp^1$-fibration, and $K_X+B+M$ is Cartier and nef, we infer that $K_X+B+M\equiv a F$ for some positive integer $a$.

\medskip

\noindent{\bf Step 2.} Now we treat the general case. Let $M'=(r-1)(K_X+B+M) + M$. Then $K_X+B+M'$ is Cartier and nef. By Step 1 we know that $N(K_X+B+M') = Nr(K_X+B+M)$ is numerically equivalent to an effective Cartier divisor for $N\geq3$.
\end{proof}

For numerically trivial generalized log canonical surfaces, we notice the following consequence of the Ascending Chain Condition (ACC) for partial minimal log discrepancies of lc surfaces (\cite{Al93}) and the Global ACC for numerically trivial generalized lc pairs (\cite{BZ16}).
\begin{thm}\label{thm: dcc}
Let $I\subset \QQ_{\geq 0}$ be a set satisfying the descending chain condition. Then there is an $N\in\NN$, depending only on $I$, such that if  $(X, B+\bfM)$ is a projective generalized lc surface with
\begin{enumerate}
\item[(i)] the coefficients of $B$ lying in $I$, 
\item[(ii)] $\bfM=\sum \mu_i \overline{M}_i$ where $M_i$ are nef Cartier divisors on $X$ and $\mu_i\in I$,
\item[(iii)] $K_X+B+\bfM_X\equiv 0$,
\end{enumerate}  
then $N(K_X+B+\bfM_X)$ is Cartier.
\end{thm}
\begin{proof}
Without loss of generality we can assume that $1\in I$. Note that $(X, B)$ is a lc surface by Lemma~\ref{lem: glc to lc}. Let $f\colon \tilde X\rightarrow X$ be a dlt modification of $(X, B)$ and  $B_{\tilde X}$ the boundary divisor on $\tilde X$ such that $K_{\tilde X} + B_{\tilde X}  = f^*(K_X+B)$. Then $\bfM_{\tilde X} = \sum \mu_i f^*M_i$ and $(\tilde X, B_{\tilde X}+\bfM)$ is still a generalized lc surface satisfying the conditions of the theorem.  Note that the Cartier index of $K_{\tilde X} + B_{\tilde X}$ is the same as that of $K_X+B$ while the Cartier indices of $\bfM_X$ and $\bfM_{\tilde X}$ differ at most by a factor $m$ such that $m\mu_i$ are integers for all $i$. Thus the uniform boundedness of the Cartier index of $K_X+B+\bfM_X$ follows from that of $K_{\tilde X}+B_{\tilde X}+\bfM_{\tilde X}$, provided that the coefficients $\mu_i$ belong to a fixed finite set.

By \cite[Theorem~3.2]{Al93}, the partial minimal log discrepancies $\mathrm{pmld}_p(\tilde X, B_{\tilde X})$  at the singular points $p\in \tilde X$ satisfy the ascending chain condition. Here  $\mathrm{pmld}_p(\tilde X, B_{\tilde X})$ is the minimal log discrepancies of exceptional divisors appearing on the minimal resolution of $p\in \tilde X$. Upon extracting the exceptional divisors $E$ over the singular points $p\in \tilde X$ such that $a_E(\tilde X, B_{\tilde X})= \mathrm{pmld}_p(\tilde X, B_{\tilde X})$, we infer that the partial minimal log discrepancies together with the coefficients of $B_{\tilde X}$ and the $\mu_i$ form a finite set $I^0$ by the Global ACC for generalized lc pairs \cite[Theorem 1.6]{BZ16}. By \cite[Lemma~3.3]{Al93}, log canonical surface singularities with fixed pmld and with coefficients from a given finite set have bounded Cartier indices. It follows that the Cartier indices of the divisors $K_{\tilde X}+B_{\tilde X}+\bfM_{\tilde X}$ under consideration are uniformly bounded. Since the $\mu_i$ appearing in the theorem belong to the fixed finite set $I^0$, we infer that the Cartier indices of $K_X+B+\bfM_X$ are also uniformly bounded.
 \end{proof}
Since going to higher birational models may reduce the Cartier indices of generalized log canonical divisors, the condition (ii) in Theorem~\ref{thm: dcc} cannot be replaced by "$\bfM=\sum \mu_i \bfM_i$ where $\bfM_i$ are nef Cartier $\rmb$-divisors on $X$ and $\mu_i\in I$", as the following example shows:

\begin{ex}
Let $f\colon \tilde X \rightarrow E$ be a ruled surface over an elliptic curve $E$ such that there is a section $E_0$ with $E_0^2 <0$. Let $E_1$ be a section of $f$ such that $E_1\cdot E_0 =0$. Set $B_{\tilde X}= E_0+E_1$.  Let $F_1$ and $F_2$ be two distinct fibres of $f$. Set $\bfM = \overline{F_1-F_2}$; by construction, $\bfM_{\tilde X} = F_1 - F_2\equiv 0$. Then $(\tilde X, B_{\tilde X}+\bfM)$ is a generalized lc pair such that  $K_{\tilde X}+B_{\tilde X}+\bfM_{\tilde X}$ is Cartier and $K_{\tilde X} +B_{\tilde X} + \bfM_{\tilde X}\equiv 0$.

Now let $h\colon \tilde X\rightarrow X$ be the contraction of $E_0$ and $B_{X} = h_* B_{\tilde X}$. Then  $K_X+B_X$ is a Cartier divisor. Note that $\bfM_X$ is $\QQ$-Cartier if and only if $(F_1 - F_2)|_{E_0}$ is a torsion line bundle on $E_0$, or, equivalently, $\bfM_{\tilde X} =F_1 - F_2$ is a torsion line bundle on $\tilde X$; moreover, the torsion order of $\bfM_{\tilde X}$ is the same as the Cartier index of $\bfM_X$. Thus, by choosing $F_1$ and $F_2$ appropriately such that $\bfM_{\tilde X}$ is a torsion of order $m$, we obtain a generalized lc surface $(X, B+\bfM)$ such that $B_X $ and $\bfM_X$ are both Weil divisors (with $\ZZ$-coefficients) and $K_X+B_X+\bfM_{X}\equiv 0$, while the Cartier index of $K_X+B_X+\bfM_{X}$ is $m$, which can be arbitrarily large.
\end{ex}

It is an interesting question as to whether Theorem \ref{thm: dcc} still holds in higher dimensions. 

\subsection{Numerical abundance for surfaces}\label{sec: num abundance}
Numerical abundance does not hold for generalized log canonical surfaces $(X, B + \bfM)$ with worse than generalized klt singularities or when $K_X + B$ is not pseudo-effective (\cite[Section 6]{LP18a}). In this subsection we give some characterization of the failure of numerical abundance in dimension two. 
\begin{thm}\label{thm: surf abund}
	Let $(X, B+\bfM)$ be a projective generalized klt surface such that $K_X+B+\bfM_X$ is nef. Then either $K_X+B+\bfM_X$ is num-semiample or $K_X+B\equiv -t\bfM_X$ for some  real number $0\leq t\leq 1$. In particular, there is a real number $0\leq t\leq 1$ such that $K_X+B+t\bfM_X$ is num-semiample.
\end{thm}
\begin{proof}
Suppose that $K_X+B+2\bfM_X$ is big. Then $$2K_X+2B+2\bfM_X-(K_X+B)=K_X+B+2\bfM_X$$ is big and nef, and by Kawamata--Shokurov's base point freeness theorem $K_X+B+\bfM_X$ is semiample.

Now we assume that $K_X+B+2\bfM_X$ is not big. Then $(K_X+B+\bfM_X)^2=0$ and $(K_X+B+\bfM_X)\cdot \bfM_X=0$. By Lemma~\ref{lem: Hodge}, the divisors $K_X+B+\bfM_X$ and $\bfM_X$ are numerically proportional. If $K_X+B+\bfM_X$ is numerically trivial, then it is num-semiample. Otherwise, there is a nonnegative real number $a$ such that $\bfM_X\equiv a(K_X+B+\bfM_X)$. If $a\geq 1$ then $K_X+B+(1-\frac{1}{a})\bfM_X\equiv 0$ is num-semiample and we take $t=1-\frac{1}{a}$. If $0\leq a<1$, then $\bfM_X\equiv \frac{a}{1-a}(K_X+B)$ and it follows that $K_X+B$ is nef and hence semiample. In this case, $K_X+B+\bfM_X\equiv \frac{1}{1-a}(K_X+B)$ is num-semiample.
\end{proof}
	
Theorem~\ref{thm: surf abund} does not hold for generalized log canonical surfaces, as the following example shows.
\begin{ex}[cf.~\cite{L04}, 2.3.A]\label{ex: counterexampledltabundance}
	Let $C_0\subset \Pp^2$ be a smooth cubic curve. Let $f\colon X\rightarrow \Pp^2$ be the blow-up of 12 general points on $C_0$, and let $C$ be the strict transform of $C_0$ on $X$. Then $K_X+C\sim 0$. Let $M=4H -E$ where $H=f^* L$ is the pull-back of a line and $E=\sum_{1\leq i\leq 12} E_i$ is the reduced exceptional locus of $f$. The divisor $M$ is big and nef, but not semiample. Since $X$ has no torsion, the numerical equivalence coincides with $\QQ$-linear equivalence for $\QQ$-divisors on $X$. Thus there is no semiample divisor that is numerically equivalent to $M\sim K_X+C+M$.
	
	Now let $h\colon \tilde X\rightarrow X$ be the blow-up of a point not lying on $C\cup E$. Let $\tilde C$ be the strict transform of $C$ and $\tilde M= h^*M-F$, where $F$ is the exceptional divisor of $f$. Since $\tilde M=4h^*H-F+\tilde{C}$. One sees easily that $\tilde M$ is nef, and $(\tilde X, \tilde C)$ is dlt.
	
Note that	$K_{\tilde X}+ \tilde C+\tilde M= h^*(K_X+C+M)$ is big and nef, but not numerically equivalent to a semiample divisor. Then for any $0\leq t <1$, we have $(K_{\tilde X}+ \tilde C+t\tilde M)\cdot F= -1+t<0.$ Thus $K_{\tilde X}+ \tilde C+t\tilde M$ is not nef, and hence not num-semiample either.
\end{ex}

Note that the surface $\tilde X$ in Example~\ref{ex: counterexampledltabundance} contains a curve $F$ such that $(K_{\tilde X} + \tilde C + t\tilde M ) \cdot F < 0 $ for any $t < 1$. Obviously, this annoying curve is contracted by a $(K_{\tilde X} +\tilde C + t\tilde M)$-MMP. We encode this observation in Theorem~\ref{thm: semiample} which deals with generalized log canonical (but not necessarily generalized klt) surfaces.

\begin{thm}\label{thm: semiample}
Let $(X, B + \bfM )$ be a generalized log canonical surface. Suppose that $K_X + B + \bfM_X$ is pseudo-effective. Then there exists a $0 \leq t_0 \leq 1$ such that $(X,B+t_0\bfM)$ has a numerically good minimal model.
\end{thm}
\begin{proof} 
	We proceed as in the proof of Theorem~\ref{thm: nv}, scaling $\bfM$ instead of $B$. Let $t_0 := \inf\{t \geq 0 \mid K_X + B + t\bfM_X \text{ is pseudo-effective}\}$.
	
	If $t_0 = 0$ then $K_X + B$ is pseudo-effective and the Abundance Theorem for log canonical surfaces gives the assertion.
	
	If $t_0 >0$ then we may run a $(K_X +B+t_0 \bfM_X)$-MMP, and reach a minimal model $(X',B_{X'} + t_0 \bfM_{X'})$, where $B_{X'}$ is the strict transform of $B$. We will show that $K_{X'}+B_{X'}+t_0 \bfM_{X'}$ is num-semiample.
	
	After running a $(K_{X'} + B_{X'} +t_1\bfM_{X'})$-MMP for some $t_1<t_0$ with $t_0-t_1$ sufficiently small, which is $(K_{X'}+B_{X'}+t_0 \bfM_{X'})$-trivial,  we reach a Mori fibre space $f\colon Y\rightarrow Z$ of $K_{Y} + B_{Y} +t_1 \bfM_{Y}$, where $B_{Y}$ is the strict transform of $B$. Then $K_{Y}+B_{Y}+t_0 \bfM_{Y}$ is the pull-back of a divisor of non-negative degree on $Z$, which is num-semiample by the fact that $\dim Z\leq 1$.
\end{proof}

 􏰂\section{Numerical nonvanishing in higher dimensions}
In this section we investigate the numerical nonvanishing for generalized polarized pairs in all dimensions. 
The generalized canonical bundle formula, made available very recently by \cite{Fi18} and extended by \cite{HLiu19}, is crucial for our purpose. 


\begin{thm}[\cite{Fi18, HLiu19}]\label{thm: cbf}
Let $\FF$ denote either the rational number field $\QQ$ or the real number field $\RR$. Let $(X/S, B+\bfM)$ be a generalized polarized pair over a quasi-projective scheme $S$ such that $\bfM$ is an $\FF_{>0}$-linear combination of nef/$S$ Cartier $\rmb$-divisors. Let $f\colon X\to Z$ be a surjective projective morphism of normal varieties, projective over $S$, such that $K_X+B+\bfM_X\sim_{\FF,f}0$. Then there is a generalized polarized pair $(Z/S, B_f+\bfM_f)$ such that $$K_X+B +\bfM_X\sim_{\FF} f^*(K_Z+B_f +\bfM_{f,Z}).$$ Moreover, if $(X/S, B+\bfM)$ is generalized lc (resp.~generalized klt), then so is $(Z/S, B_f+\bfM_f)$.
\end{thm}

\begin{lem}\label{lem: gmm}
Let $(X/S, B+\bfM)$ be a generalized klt pair. Then the following holds.
\begin{enumerate}
\item[(i)] The singularities of $X$ are rational.
\item[(ii)] If $K_X+B+\bfM_X$ is pseudo-effective/$S$ and either $B$ or $\bfM_X$ is big/$S$, then $(X/S, B+\bfM)$ has a good minimal model over $S$. 
\item[(iii)] If $K_X+B+\bfM_X$ is not pseudo-effective/$S$, then we may run a $(K_X+ B+\bfM_X)$-MMP with scaling of an ample/$S$ $\RR$-Cartier $\RR$-divisor and end with a Mori fiber space over $S$.  
\end{enumerate}
\end{lem}
\begin{proof}
If $B$ or $\bfM_X$ is big/$S$ then one can find a big/$S$ boundary divisor $\Delta\sim_{\RR, S} B+M$ such that $(X, \Delta)$ is klt; see \cite[Lemma~3.5]{HL18}. 

(i) Let $H$ be an ample divisor over $S$ on $X$. Replacing $\bfM$ with $\bfM + \overline{H}$, we can assume that $\bfM_X$ is big over $S$. Thus there is a klt pair $(X, \Delta)$ as above and the assertion follows from the fact that klt singularities are rational.

(ii) By \cite{BCHM10}, the pair $(X, \Delta)$ found above has a good minimal model over $S$ which is automatically one for $(X/S, B+\bfM)$.

(iii) Let $H$ be an ample/$S$ $\Rr$-divisor on $X$ such that $K_X+B+H+\bfM_X$ is not pseudo-effective. Replacing $\bfM$ with $\bfM + \overline{H}$, we can assume that $\bfM_X$ is big. Thus there is a klt pair $(X, \Delta)$ as above. By \cite{BCHM10}, we may run a $(K_X+\Delta)$-MMP with scaling of $H$ and end with a Mori fiber space over $S$, which is also one for $(X/S, B+\bfM)$.
\end{proof}

We generalize a construction for log canonical pairs (Lemmas \ref{lem: gongyo3.1} and  \ref{lem: gmm2}), initiated in \cite[Proposition 8.7]{DHP13} and then extended and applied in \cite{Bi12,G15, DL15, Has17, Has18}. The content of the lemmas should be known to experts, and the ideas of the proofs are already contained in the above references: we scale a non-pseudo-effective generalized log canonical divisor $K_X+B'+\bfM'_X$ with another generalized boundary divisor $B''+\bfM''_X$, in order to construct a Mori fibre space $Y\rightarrow Z$, and then one can run MMP over $Y$ and $Z$, if needed. 
\begin{lem}\label{lem: gongyo3.1}
Let $(X/S,(B'+B'')+(\bfM'+\bfM''))$ be a $\Qq$-factorial generalized dlt pair, where $\bfM'$ and $\bfM''$ are $\Rr_{> 0}$-linear combinations of nef/$S$ Cartier $\rmb$-divisors. Suppose that $K_X+(B'+B'')+(\bfM'_X+\bfM''_X)$ is pseudo-effective/$S$  but $K_X+B'+\bfM'_X$ is not. Let $$t_0:=\inf\{t \geq 0\mid K_X+B'+\bfM'_X+t(B''+\bfM''_X) \text{ is pseudo-effective/}S\}.$$ Then there exists a  birational contraction $\phi\colon X \dashrightarrow Y$ such that there exists a projective morphism $f\colon Y\to Z$ with connected fibers satisfying:
	
	(i) $(Y/S,(B'_Y+t_0 B''_Y)+(\bfM'_Y+t_0 \bfM''_Y))$ is $\Qq$-factorial generalized lc,
	
    (ii) $(Y/S,B'_Y+\bfM'_Y)$ is generalized klt,
	
	(iii) the relative Picard number $\rho(Y/Z)=1$,
	
	(iv) $K_{Y}+ (B'_Y+t_0 B''_Y)+(\bfM'_Y+t_0 \bfM''_Y)\sim_{\RR, f} 0$,
	
	(v) $B''_Y+\bfM''_Y$ is ample over $Z$, and
	
	(vi) $\dim Y>\dim Z$,\\	
	where $B'_Y,B''_Y$ are the strict transforms of $B',B''$ on $Y$ respectively. In particular, if $B',\bfM'_X,B'',\bfM''_X$ are $\Qq$-divisors then $t_0\in\Qq$.
\end{lem}
\begin{proof}
Since $K_X+(B'+B'')+(\bfM'_X+\bfM''_X)$ is pseudo-effective/$S$  and $K_X+B'+\bfM'_X$ is not pseudo-effective/$S$ , we have $0<t_0\leq 1$. Let $t_i$ be a strictly increasing sequence of positive real numbers $t_i$ such  that $\lim_{i\to +\infty}t_i=t_0$. By \cite[Lemma~4.4]{BZ16}, we may run an MMP of $K_X+(B'+t_iB'')+(\bfM'_X+t_i\bfM''_X)$ with scaling of an ample divisor and reach a Mori fiber space $f_i:Y_i\to Z_i$. Since $\rho(Y_i/Z_i)=1$, there exists a positive number $\eta_i$ such that
	$$K_{Y_i}+(B'_{Y_i}+\eta_iB''_{Y_i})+(\bfM'_{Y_i}+\eta_i\bfM''_{Y_i})\equiv_{Z_i} 0,$$
	where $B'_{Y_i},B''_{Y_i}$ are the strict transforms of $B',B''$ on $Y_i$.
	
	It is clear that $t_i<\eta_i\le t_0$, $\lim_{i\to+\infty}\eta_i=t_0$, and $t_i\le\glct(Y_i/S,B'_{Y_i}+\bfM'_{Y_i};B''_{Y_i}+\bfM''_{Y_i})$, where $\glct$ denotes the generalized log canonical threshold. By \cite[Theorem 1.5]{BZ16}, $(Y_i/S,(B'_{Y_i}+t_0 B''_{Y_i})+(\bfM'_{Y_i}+t_0 \bfM''_{Y_i}))$ and hence $(Y_i/S,(B'_{Y_i}+\eta_i B''_{Y_i})+(\bfM'_{Y_i}+\eta_i \bfM''_{Y_i}))$ are generalized lc  for $i\gg 0$. 
		
	Let $F_i$ be a general fiber of $f_i$. Then
	$$K_{F_i}+(B'_{F_i}+\eta_i B''_{F_i})+(M'_{F_i}+\eta_i M''_{F_i}):=(K_{Y_i}+(B'_{Y_i}+\eta_i B''_{Y_i})+(\bfM'_{Y_i}+\eta_i \bfM''_{Y_i}))|_{F_i}\equiv_S0.$$
	By the Global ACC (\cite[Theorem 1.6]{BZ16}), there is some $i_0$ such that $\eta_i=t_0$ for $i\geq i_0$. Setting $Y=Y_{i_0}$, $Z=Z_{i_0}$, we obtain the desired birational map $X\dashrightarrow Y$. 
	
If $B',B'',\bfM'_X,\bfM''_X$ are $\Qq$-divisors, then so are $B'_Y,B''_Y,\bfM'_Y,\bfM''_Y$. It follows from (iii) and (iv) that $t_0\in\Qq$.
\end{proof}

In order to use the generalized canonical bundle formula for the purpose of induction on dimension, we need a (relative) good minimal model whose (relative) Iitaka fibration has positive dimensional fibres.
\begin{lem}\label{lem: gmm2}
Let the notation be as in Lemma~\ref{lem: gongyo3.1}. Assume additionally that $(X/$S$,(B'+B'') + (\bfM'+\bfM''))$ is generalized klt. Then there is a higher model $(W/$S$, B_W+ (\bfM'+t_0 \bfM''))$ of  $(X/$S$,(B'+t_0 B'') + (\bfM'+t_0 \bfM''))$ admitting a good minimal model over $Z$ whose relative Iitaka fibration has positive dimensional fibres. Moreover, if $(X/$S$,(B'+B'') + (\bfM'+\bfM''))$ has $\QQ$-coefficients then so does $(W/$S$, B_W+(\bfM'+t_0 \bfM''))$.
\end{lem}
\begin{proof}
Upon replacing $B''$ and $\bfM''$ with $t_0B''$ and $t_0\bfM''$ respectively, we may assume that $t_0=1$.  
To simply  the notation further, we denote $B = B'+B''$ and $\bfM = \bfM'+\bfM''$.

Let $p:W\to X$ be a log resolution to which the $\rmb$-$\RR$-divisor $\bfM$ descends. We may assume that the induced birational map $q:W\dashrightarrow Y$ is a morphism. (Note that we are using the notation of Lemma~\ref{lem: gongyo3.1}.) Let $B_Y = B'_Y+B''_Y$. Since $B''_Y+\bfM''_Y$ is ample over $Z$ and $\rho(Y/Z) = 1$,  the divisor $B_Y + \bfM_Y$ is also ample over $Z$. Again by the fact that $\rho(Y/Z) = 1$, each of $B_Y$ and $ \bfM_Y$ is ample over $Z$ as soon as it is not numerically trivial over $Z$.

 Since $(X/$S$,B+\bfM)$ is generalized klt, we may choose a rational number $0<\epsilon\ll1$ and $B_W=p_{*}^{-1} B+(1-\epsilon)F_W$, $F_W$ being the sum of the reduced exceptional divisors over $X$, such that $$K_W+B_W+\bfM_W=p^{*}(K_X+B+\bfM_X)+G_W,$$ where $G_W$ is an $p$-exceptional effective $\RR$-divisor. Then the generalized klt pair $(W/S, B_W + \bfM)$ is a higher model of $(X/S,B + \bfM)$. We will show that the former has a good minimal model over $Z$.

By Lemma~\ref{lem: gmm}, $(W/$S$, B_W+\bfM)$ has a canonical model $(U/$S$, B_U+\bfM)$ over $Y$, where $B_U$ is the strict transform of $B_W$. We have $K_{U}+B_U+\bfM_U+E_U=g^{*}(K_{Y}+B_Y+\bfM_Y),$ where $g\colon U\rightarrow Y$ is the induced morphism and $E_U$ is some $g$-exceptional divisor. Since $-E_U$ is ample over $Y$, $E_U$ is effective and $\Supp(E_U) = \Exc(g)$.

We can choose two rational numbers $0<\delta_1, \delta_2\ll 1$ such that 
\begin{itemize}[leftmargin=*]
\item $(U/S, (B_U +\delta_1 E_U-\delta_2 (g^* \bfM_Y -M_U))+((1-\delta_2 ) \bfM +\delta_2  \overline\bfM_Y))$ is generalized klt, and
\item the boundary part $B_U +\delta_1 E_U-\delta_2 (g^* \bfM_Y -\bfM_U) \geq \epsilon g^* B_Y$ for some $0<\epsilon \ll1$.
\end{itemize}
Since at least one of $B_Y$ and $\bfM_Y$ is ample over $Z$, one sees easily that either $B_U +\delta_1 E_U+ \delta_2 (g^* \bfM_Y -M_U)$ or $(1-\delta_2) \bfM_U +\delta_2 g^* \bfM_Y$ is big over $Z$. By Lemma~\ref{lem: gmm}, $(U, (B_U +\delta_1 E_U-\delta_2 (g^* \bfM_Y -\bfM_U))+((1-\delta_2 ) \bfM +\delta_2\overline{\bfM}_Y))$ has a good minimal model $U\dashrightarrow V$ over $Z$.

Note that 
\begin{multline}
K_U+ (B_U +\delta_1 E_U - \delta_2 (g^* \bfM_Y -\bfM_U))+((1-\delta_2 ) \bfM_U +\delta_2 g^* \bfM_Y) \\ = K_U+B_U+\bfM_U +\delta_1 E_U
\end{multline}
Since  $K_{U}+B_U+\bfM_U+E_U\sim_{\RR,Z}0$, we have
	\begin{align*}
	(\frac{1}{\delta_1}-1)(K_{U}+B_U+\bfM_U)&\sim_{\RR,Z} \frac{1}{\delta_1}(K_{U}+B_U+\bfM_U)+E_U\\
	&\sim_{\RR,Z}\frac{1}{\delta_1}(K_{U}+B_U+\delta_1 E_U+ \bfM_U).
	\end{align*}
It follows that the $(K_{U}+B_U+\delta_1 E_U+\bfM_U)$-MMP is at the same time a $(K_{U}+B_U+\bfM_U)$-MMP, and $(V, B_V+ \bfM)$ is a good minimal model of $(U, B_U+\bfM)$ over $Z$, where $B_V$ is the strict transform of $B_U$. Obviously, $K_V+B_V+\bfM_V$ is not big over $Z$, and its relative Iitaka fibration has positive dimensional fibres.

The rational maps involved in the proof are as in the following commutative diagram:
$$\xymatrix@=2.5em{
	& W \ar[d]^{q}  \ar[dl]_{p}  \ar@{-->}[r]   & U\ar[dl]^{g}  \ar@{-->}[r]    &V  \ar[ddll]\\
	X\ar@{-->}[r]  &Y  \ar[d]^{f}     &    \\
	& Z       &&
}
$$
\end{proof}

Now we are ready to prove several numerical nonvanishing results in dimensions higher than two. 
\begin{thm}\label{thm: highdim}
Assume that $\text{Conjecture~\ref{conj: nonvanishing}}$ holds for projective generalized klt pairs in dimensions less than $n$. Let $(X,B+\bfM)$ be an $n$-dimensional projective generalized klt pair such that $\bfM$ is an $\RR_{>0}$-linear combination of nef Cartier $\rmb$-divisors. Suppose that $K_X+B+\bfM_X$ is pseudo-effective and $K_X+\bfM_X$ is not pseudo-effective. Then $K_X+B+\bfM_X$ is num-effective.
\end{thm}
\begin{proof}
By Lemma~\ref{lem: gmm}, $X$ has rational singularities. Concerning numerical nonvanishing, we can replace $X$ with a higher model and assume that $(X,B)$ is a $\QQ$-factorial klt pair. Let $t_0:=\inf\{t\geq 0\mid K_X+tB+\bfM_X\text{ is pseudo-effective}\}$. By Lemma~\ref{lem: gmm2}, there is a higher model $(W, B_W+\bfM)$ of $(X, t_0B + \bfM)$ admitting a good minimal model $(V, B_V+\bfM)$ over some variety $Z$, where $B_V$ is the strict transform of $B_W$. 
	
Let $h\colon V\rightarrow \tilde Z$ be the relative Iitaka fibration of $K_V+B_V+\bfM_V$. Then $\dim \tilde Z < \dim V$.  By Theorem~\ref{thm: cbf} there is a generalized klt pair $({\tilde Z}, B_h+\bfM_h)$ such that $K_V+B_V+\bfM_V\sim_{\RR}  h^*(K_{\tilde Z}+B_h+\bfM_{h, \tilde Z})$. Note that $K_{\tilde Z}+B_h+\bfM_{h,\tilde Z}$ is necessarily pseudo-effective. Since $\dim \tilde Z<\dim X$, we know by hypothesis that $K_{\tilde Z}+B_{h}+\bfM_{h, \tilde Z}$ is num-effective. Therefore, $K_V+B_V+\bfM_V$ is num-effective. It follows that $K_W+B_W+\bfM_W$ is num-effective. Pushing down to $X$ we deduce that $K_X+t_0B+\bfM_X$, and a fortiori $K_X+B+\bfM_X$, is num-effective.
\end{proof}

Assuming the termination of flips for $\Qq$-factorial dlt pairs, Theorem~\ref{thm: highdim} readily extends to generalized lc pairs with rational singularities as follows:
\begin{thm}\label{thm: highdim dlt}
Assume that $\text{Conjecture~\ref{conj: nonvanishing}}$ holds in dimensions less than $n$. 
Assume that the termination of flips for $n$-dimensional projective $\Qq$-factorial dlt pairs holds. 

Let $(X,B+\bfM)$ be an $n$-dimensional projective generalized lc pair with rational singularities such that $\bfM$ is an $\RR_{>0}$-linear combination of nef Cartier $\rmb$-divisors. Suppose that $K_X+B+\bfM_X$ is pseudo-effective and $K_X+\bfM_X$ is not pseudo-effective. Then there exists an effective $\RR$-Cartier $\RR$-divisor $D$ such that $K_X+B+\bfM_X\equiv D$.
\end{thm}
\begin{proof}
Since $(X, B+\bfM)$ has rational singularities, its numerical nonvanishing is implied by that of its $\QQ$-factorial dlt modification (\cite[Lemma~2.12]{Na04}). Thus we can replace $(X,B+\bfM)$ with a $\QQ$-factorial dlt modification and assume that itself is $\Qq$-factorial generalized dlt.
	
Let $t_0=\inf\{t\ge0\mid K_X+tB+\bfM_X \text{ is pseudo-effective}\}$. Then $0<t_0\leq 1$. Since we assume the termination of flips for $n$-dimensional $\Qq$-factorial dlt pairs, by \cite[Theorem 1.6]{HL18}, $K_X+t_0B+\bfM_X$ has a log terminal model, $\phi: X\dashrightarrow X'$. In particular, $\phi$ is $(K_X+t_0B+\bfM_X)$-negative. Thus there exists a positive real number $\epsilon>0$ such that, for any $t\in[t_0-\epsilon,t_0)$, $\phi$ is $(K_X+tB+\bfM_X)$-negative; obviously, $K_{X'}+tB_{X'}+\bfM_{X'}$ is not pseudo-effective, where $B_{X'}$ is the strict transform of $B$, and $(X,tB+\bfM)$ is generalized klt. 

By \cite[Lemma 3.17]{HL18}, there exists a positive real number $t_1\in (t_0-\epsilon,t_0)$, such that any $(K_{X'}+t_1B_{X'}+\bfM_{X'})$-MMP is $(K_X+t_0B_{X'}+\bfM_{X'})$-trivial. By Lemma \ref{lem: gmm}, we may run a $(K_{X'}+t_1B_{X'}+\bfM_{X'})$-MMP with scaling of an ample divisor and end with a Mori fiber space, $f:Y\to Z$ of $K_Y+t_1B_Y+\bfM_{Y}$, where $B_Y$ is the strict transform of $B$. By Theorem \ref{thm: cbf}, $K_Y+t_0B_Y+\bfM_{Y}\sim_{\RR} f^*(K_{Z}+t_0B_f+\bfM_{f, Z})$ for some generalized lc pair $({Z}, t_0B_f+\bfM_f)$. Since we assume $\text{Conjecture~\ref{conj: nonvanishing}}$ holds in dimensions less than $n$, $K_Y+t_0B_Y+\bfM_{Y}$ is num-effective. Hence $K_X+t_0B+\bfM_{X}$ is also num-effective.
\end{proof}

\begin{cor}\label{thm: CYnonvanishing threefold}
	Let $(X, B+\bfM)$ be a generalized lc threefold with rational singularities such that $\bfM$ is an $\RR_{>0}$-linear combination of nef Cartier $\rmb$-divisors. Suppose that
	\begin{enumerate}[leftmargin=*]
		\item[(i)]$K_X+B+\bfM_X$ is pseudo-effective, and
		\item[(ii)] there exists a generalized lc pair $(X, C+\bfM)$ with $K_X+C+\bfM_X\equiv 0$.
	\end{enumerate}
	Then $K_X+B+\bfM_X$ num-effective.
\end{cor}
\begin{proof}
If $C=0$, then $K_X+B+\bfM_X\equiv B$ which is effective. We may assume $C\neq 0$, so $K_X+\bfM$ is not pseudo-effective. Then the assertion of the corollary follows from Theorem \ref{thm: highdim dlt}.
\end{proof}

Up to scaling the nef part we establish the numerical nonvanishing for generalized lc threefolds with rational singularities:
\begin{thm}\label{thm: dim3t}
Let $(X,B+\bfM)$ be a generalized lc threefold with rational singularities such that $\bfM$ is an $\RR_{>0}$-linear combination of nef Cartier $\rmb$-divisors. Suppose that $K_X+B+\bfM_X$ is pseudo-effective. Then $K_X+B+t_0\bfM_X$ is num-effective, where $t_0:=\inf\{t\geq 0\mid K_X+B+t\bfM_X \text{ is pseudo-effective}\}.$
\end{thm}

\begin{proof}
	If $t_0=0$, then $K_X+B$ is pseudo-effective, and the theorem follows from the nonvanishing theorem for projective lc threefolds (\cite{Sho96,KMM94}).
	
	If $t_0>0$, then $K_X +B$ is not pseudo-effective. Replacing $(X,B+\bfM)$ with a $\QQ$-factorial generalized dlt modification, we can assume that $(X,B+\bfM)$ is itself a $\QQ$-factorial generalized dlt threefold. By the termination of lc threefolds (\cite{Sho96}),  $K_X+B+t_0\bfM_X$ has a log terminal model $\phi: X\dashrightarrow X'$ (\cite[Theorem 1.6]{HL18}).
 In particular, $\phi$ is $(K_X+B+t_0\bfM_X)$-negative. Thus there exists a positive real number $\epsilon>0$, such that $\phi$ is $(K_X+B+t\bfM_X)$-negative and $(K_{X'}+B_{X'}+t\bfM_{X'})$ is not pseudo-effective for any $t\in[t_0-\epsilon,t_0)$, where $B_{X'}$ is the strict transform of $B$. 
	
	By \cite[Lemma 3.18]{HL18}, there exists a positive real number $t_1\in (t_0-\epsilon,t_0)$, such that any $(K_{X'}+B_{X'}+t_1\bfM_{X'})$-MMP is $(K_X+B_{X'}+t_0\bfM_{X'})$-trivial. By \cite[Lemma 3.5]{HL18} and \cite{BCHM10}, we may run a $(K_{X'}+B_{X'}+t_1\bfM_{X'})$-MMP with scaling of an ample divisor and end with a Mori fiber space, $f:Y\to Z$ of $K_Y+B_Y+t_1\bfM_{Y}$, where $B_Y$ is the strict transform of $B$. By Theorem~\ref{thm: cbf} there is a generalized lc pair $({Z}, B_f+\bfM_f)$, such that $K_Y+B_Y+t_0\bfM_Y\sim_{\RR} f^*(K_{Z}+t_0B_f+\bfM_{f, Z})$. Since $K_{Z}+B_{f}+\bfM_{f, Z}$ is necessarily pseudo-effective and $\dim \tilde Z\leq 2$, the divisor $K_{Z}+B_{f}+\bfM_{f, Z}$ is num-effective by Theorem~\ref{thm: nv}. Therefore, $K_Y+B_Y+t_0\bfM_Y$ is num-effective. Hence $K_X+B+t_0\bfM_{X}$ is also num-effective.
\end{proof}

\end{document}